\documentclass[10pt,reqno,final]{amsart}
\usepackage{epsfig,amssymb,amsmath,version}
\usepackage{amssymb,version,graphicx,fancybox,mathrsfs,multirow}
\usepackage[notcite,notref]{showkeys}
\usepackage{url,hyperref}
\usepackage{subfigure}
\usepackage{color}
\usepackage{cases}

\catcode`\@=11 \theoremstyle{plain}
\@addtoreset{equation}{section}   
\renewcommand{\theequation}{\arabic{section}.\arabic{equation}}
\@addtoreset{figure}{section}
\renewcommand\thefigure{\thesection.\@arabic\c@figure}
\newtheorem{thm}{\bf Theorem}

\@addtoreset{thm}{section}

\newenvironment{theorem}{\begin{thm}} {\end{thm}}
\newtheorem{cor}{\bf Corollary}
\newtheorem{prop}{Proposition}[section]

\newtheorem{lmm}{\bf Lemma}

\theoremstyle{remark}
\newtheorem{rem}{Remark}[section]

\numberwithin{table}{section}

\def \af {\alpha}
\def \bt {\beta}

\renewcommand  \varphi \phi

\newcommand{\bs}[1]{\boldsymbol{#1}}

\renewcommand \wedge \times

\begin{document}
\bibliographystyle{plain}
\baselineskip 14pt

{\title[Collocation methods and Birkhoff interpolation] {A well-conditioned collocation method using pseudospectral  integration matrix}
\author[L. Wang,\;\;\; M. Samson\;\;\;  $\&$\;\;  X. Zhao] {Li-Lian Wang,\;  Michael Daniel Samson\;
and\;
 Xiaodan Zhao}
\thanks{\noindent Division of Mathematical Sciences, School of Physical
and Mathematical Sciences,  Nanyang Technological University,
637371, Singapore. The research of the authors is partially supported by Singapore MOE AcRF Tier 1 Grant (RG 15/12), and Singapore A$^\ast$STAR-SERC-PSF Grant (122-PSF-007).
}

\begin{abstract} In this paper, a well-conditioned collocation method is constructed for solving general $p$-th
order linear differential equations with various types of boundary conditions.
Based on a suitable Birkhoff interpolation, 
we obtain a new set of polynomial basis functions that results in a collocation scheme
with two important features: the condition number of the linear system is  independent of the number of collocation points;
and  the underlying  boundary conditions are imposed exactly.
Moreover, the new basis leads to exact inverse of the pseudospectral differentiation matrix (PSDM)   of
the highest derivative (at interior collocation points), which is therefore called the  pseudospectral integration matrix (PSIM).
We show that PSIM  produces the optimal integration preconditioner, and  stable collocation solutions
 with even thousands of points.
\end{abstract}
 \maketitle

\vspace*{-15pt}

\section{Introduction}\label{sect1}

The spectral collocation method is implemented in  physical space, and approximates derivative values  by
direct differentiation of the Lagrange interpolating polynomial at a set of Gauss-type  points.
Its fairly straightforward realization is akin to the high-order finite difference method (cf.  \cite{Forn96,Tref00}).
This marks its advantages over the spectral method using modal basis functions
in dealing with variable coefficient and/or nonlinear problems (see various monographs on spectral methods
\cite{gottlieb1977numerical,Guobk98,Boyd01,CHQZ06,HGG07,ShenTangWang2011}).
 However, the practitioners  are plagued with the involved ill-conditioned linear systems
 (e.g., the condition number of the $p$-th order differential operator grows  like $N^{2p}$).
 This longstanding drawback causes  severe degradation of expected spectral accuracy \cite{TrTr87}, while the
 accuracy of machine zero  can be well observed from the well-conditioned spectral-Gakerkin method (see e.g., \cite{Shen94b}).
 In practice,  it becomes rather prohibitive to solve the  linear system by a direct solver or even an iterative method,
 when the number of collocation points is large.

One significant attempt  to circumvent this barrier is the use of  suitable {\em preconditioners.}
Preconditioners built on low-order  finite difference or finite element approximations can be found in e.g., \cite{Dev.M85,Dev.M90,CA85,Kim.P96,Kim.P97,CA10}.
The {\em integration preconditioning} (IP) proposed by Coutsias, Hagstrom and Hesthaven et al. \cite{CHT96,coutsias1996integration,Hesthaven98} (with ideas from Clenshaw  \cite{clenshaw1957numerical})
 has proven to be  efficient. 
We highlight that  the IP in  Hesthaven \cite{Hesthaven98} led to a significant reduction of the condition number from $O(N^2)$
to  $O(\sqrt N)$ for second-order differential linear operators with Dirichlet boundary conditions (which were imposed by the penalty method \cite{FG88}).
Elbarbary \cite{Elbarbary06} improved the IP in \cite{Hesthaven98} through carefully  manipulating the involved singular matrices and imposing the boundary conditions by some auxiliary equations.
 Another remarkable  approach is the {\em spectral integration method} proposed by Greengard  \cite{Greengard91}  (also see
\cite{Zebib84}), which recasts the differential form into integral  form, and then approximates
 the solution by orthogonal polynomials. This method was incorporated into the {\tt chebop} system \cite{DBT08,Dris10}.
 A  relevant approach by El-Gendi \cite{El69} is without reformulating the differential equations, but
 uses the integrated Chebyshev polynomials as basis functions. Then the spectral integration matrix (SIM) is employed in place of
PSDM to
 obtain much better conditioned linear systems (see e.g., \cite{MM02,ghoreishi2004tau,Muite10,ElS13} and the references therein).

In this paper, we take a very different routine  to  construct well-conditioned collocation methods.  The essential idea
is to associate the highest differential operator and  underlying boundary conditions with a suitable
{\em  Birkhoff interpolation} (cf. \cite{BirkhoffBk,shi2003theory}) that interpolates the derivative values at  interior collocation
points, and interpolate the boundary data at  endpoints.  This leads to the so-called
{\em Birkhoff interpolation basis polynomials} with the following distinctive  features:
\vspace*{-1pt}
\begin{itemize}
\item [(i)] Under the new basis, the linear system of a usual collocation scheme is well-conditioned, and
  the matrix of the highest derivative is diagonal or identity. Moreover, the underlying boundary conditions are imposed exactly.
    This technique can be viewed as the collocation analogue of the well-conditioned spectral-Galerkin method
    (cf. \cite{Shen94b,Shen02b,Guo.SW06}) (where
     the matrix of the highest derivative in the Galerkin system is  diagonal under certain modal  basis functions).
\item [(ii)] The new basis  produces the {\em exact inverse} of PSDM of the highest derivative
(involving only  interior collocation points). This inspires us to introduce the concept of
pseudospectral integration matrix (PSIM). The integral expression of the new basis
 offers a stable way to compute  PSIM and the inverse of PSDM even for thousands of collocation points.

\item [(iii)] This  leads to optimal integration preconditioners
for the usual collocation methods, and enables us  to have insights into the  IP in \cite{Hesthaven98,Elbarbary06}.
Indeed, the preconditioning  from Birkhoff interpolation is natural and optimal.
\end{itemize}

 We point out that Castabile and Longo \cite{CoL10} touched on the application of Birkhoff interpolation (see \eqref{Birkhoffb})
 to  second-order boundary value problems (BVPs),
 but the focus of this work was largely on the analysis of  interpolation  and quadrature errors.
  Zhang  \cite{ZhangInp2012} considered the Birkhoff interpolation (see \eqref{Birkhoffa}) in  a very different context of superconvergence of polynomial
  interpolation.  Collocation methods based on a special Birkhoff quadrature rule  for Neumann problems were discussed in \cite{EZGu99,WangG09}.
   It is also noteworthy  to point out recent interest in developing  spectral solvers using modal basis functions
   (see e.g., \cite{LivP10,ChenS12,OlverTownsend}).

The rest of the paper is organized as follows.  In Section \ref{sect2},
we review several topics that are pertinent to the forthcoming development.
 In Section \ref{sect3}, we elaborate on the new methodology for second-order BVPs.
 In Section \ref{sect4},  we present miscellaneous extensions of the approach to first-order initial value problems (IVPs), higher order equations
 and multiple dimensions.

\section{Birkhoff interpolation and pseudospectral differentiation matrix} \label{sect2}

In this section, we briefly review several topics  directly bearing on the subsequential algorithm and analysis.
We also introduce the notion of pseudospectral integration matrix, which is a central piece of puzzles for our new approach.

\subsection{Birkhoff interpolation}
Let  $\{x_j\}_{j=0}^{N}\subseteq [-1,1]$ be a set of distinct interpolation points, which are  arranged  in ascending order:
 \begin{equation}\label{interpoint}
-1\le x_0<x_1<\cdots<x_{N-1}<x_N\le 1.
\end{equation}
Given $K+1$  data $\{y_j^m\}$ (with $K\ge N$),  we consider the interpolation problem (cf. \cite{BirkhoffBk,shi2003theory}):
\begin{equation}\label{birkhoffprob}
\left\{\begin{aligned}
&\text{Find a polynomial $p_K\in {\mathbb P}_K$  such that}\\
&p_K^{(m)}(x_j)=y_j^m \quad   \text{($K+1$ equations)},
\end{aligned}\right.
\end{equation}
where $ {\mathbb P}_K$ is the set of all algebraic polynomials of degree at most $K,$ and
the subscript $m$ indicates the order of specified derivative values.

We have the   Hermite interpolation if for each $j,$ the orders of derivatives in
\eqref{birkhoffprob} form an unbroken sequence, $m=0,1,\cdots, m_j.$ In this case, the interpolation polynomial $p_K$  uniquely exists and can be given by an explicit formula.
 On the other hand, if some of the sequences are broken, we have the {\em Birkhoff interpolation.} However,
 the existence and uniqueness of the Birkhoff interpolation polynomial are not guaranteed.
For example, for  \eqref{birkhoffprob} with $K=N=2,$ and the given data $\{y_0^0, y_1^1, y_2^1\}$,
 the quadratic polynomial $p_2(x)$ does not exist, when $x_1=(x_0+x_2)/2.$
 This happens to Legendre/Chebyshev-Gauss-Lobatto points, where $x_0=-1,x_1=0$ and $x_2=-1$.
We refer to the monographs  \cite{BirkhoffBk,shi2003theory} for comprehensive discussions of Birkhoff interpolation.

 In this paper, we will consider special Birkhoff interpolation  problems at Gauss-type points, and
some variants that incorporate with mixed boundary data, for instance, $a p'_K(-1)+bp_K(-1)=y_0$  for constants $a,b.$

\subsection{Pseudospectral differentiation matrix}
The pseudospectral differentiation matrix (PSDM) is an essential  building block for collocation methods.
Let $\{x_j\}_{j=0}^N$ (with $x_0=-1$ and $x_N=1$)  be a set of Gauss-Lobatto (GL) points, and let
  $\{l_j\}_{j=0}^N$ be the Lagrange interpolation basis polynomials such that
  $ l_j\in {\mathbb P}_N $ and $l_j(x_i)=\delta_{ij},$ for $ 0\le i,j\le N.$
We have
  \begin{equation}\label{diffpn}
  p(x)=\sum_{j=0}^N p(x_j) l_j(x),\quad \forall p\in {\mathbb P}_N.
  \end{equation}
Denoting  $d_{ij}^{(k)}:=l_{j}^{(k)}(x_i),$ we introduce the matrices
\begin{equation}\label{DkDkdefn}
 \bs D^{(k)}=\big(d_{ij}^{(k)}\big)_{0\le i,j\le N},\quad  \bs D^{(k)}_{\rm in}=\big(d_{ij}^{(k)}\big)_{1\le i,j\le N-1},\quad k\ge 1.
 \end{equation}
 Note that $\bs D^{(k)}_{\rm in}$ is obtained by deleting the last and first rows and columns of $\bs D^{(k)},$
 so it is associated with  interior GL points.  In particular, we denote  $\bs D=\bs D^{(1)},$ and $\bs D_{\rm in}=\bs D^{(1)}_{\rm in}$. The matrix $\bs D^{(k)}$ is usually referred to as the $k$-th order PSDM.
 We highlight  the following property (see e.g.,  \cite[Theorem 3.10]{ShenTangWang2011}):
 \begin{equation}\label{DkDk}
\bs D^{(k)}=\bs D\bs D\cdots \bs D=\bs D^k,\quad k\ge 1,
 \end{equation}
so  higher-order PSDM is a product of the first-order PSDM.

Set
\begin{equation}\label{diffps}
\bs p^{(k)}:=\big(p^{(k)}(x_0),\cdots, p^{(k)}(x_N)\big)^t,  \quad \bs p:=\bs p^{(0)}.
\end{equation}
By \eqref{diffpn} and \eqref{DkDk}, the pseudospectral differentiation process is performed via
\begin{equation}\label{diffpr}
\bs D^{(k)} \bs p=\bs D^{k} \bs p=\bs p^{(k)},\quad k\ge 1.
\end{equation}
It is noteworthy that differentiation via \eqref{diffpr} suffers from
significant round-off errors  for large $N,$ due to the involvement of ill-conditioned operations (cf. \cite{weideman2000matlab}).
The matrix  $\bs D^{(k)}$ is singular (a simple proof:  $\bs D^{(k)}\bs 1=\bs 0,$ where $\bs 1=(1,1,\cdots, 1)^t,$
so the rows of $\bs D^{(k)}$ are linearly dependent), while  $\bs D^{(k)}_{\rm in}$ is nonsingular.
In addition, the condition numbers of $\bs D^{(k)}_{\rm in}$  and $\bs D^{(k)}-\bs I_{N+1}$  behave like $O(N^{2k}).$
We refer to  \cite[Section 4.3]{CHQZ06} for review of eigen-analysis for PSDM.


\subsection{Legendre and Chebyshev polynomials} We collect below some properties of Legendre and Chebyshev polynomials
(see e.g., \cite{szeg75,ShenTangWang2011}), to be used throughout this paper.

Let  $P_k(x), x\in I:=(-1,1)$ be the Legendre polynomial of degree $k.$  The Legendre polynomials are mutually orthogonal:
\begin{equation} \label{Legorth}
\int_{-1}^1 P_k(x) P_j(x)\, dx=\gamma_k  \delta_{kj}, \quad \gamma_k=\frac 2{2k+1}.
\end{equation}
There hold
\begin{align}\label{legretcff}
P_k(x)=\frac{1} {2k+1}\big(P_{k+1}'(x)-P_{k-1}'(x)\big),\quad k\ge 1,
\end{align}
and
\begin{align} \label{legpm1}
 P_k(\pm 1)=(\pm 1)^k,\quad P_k'(\pm 1)=\frac 1 2 (\pm 1)^{k-1} k(k+1).
\end{align}

The Legendre-Gauss-Lobatto (LGL) points
 are zeros of  $(1-x^2)P_N'(x),$ and the corresponding  quadrature weights are
 \begin{equation}\label{LGLweights}
   \omega_j=\frac 2 {N(N+1)}\frac {1}{P_N^2(x_j)},\quad 0\le j\le N.
 \end{equation}
Then  the LGL quadrature has the exactness
\begin{equation}\label{quadexac}
\int_{-1}^1 \phi(x) dx=\sum_{j=0}^N \phi(x_j)\omega_j,\quad \forall \phi\in {\mathbb P}_{2N-1}.
\end{equation}

The Chebyshev polynomials: $T_k(x)=\cos(k\, {\rm arccos}(x))$ are mutually orthogonal
\begin{equation} \label{Chebyorth}
 \int_{-1}^1 \frac{T_k(x) T_j(x)}{\sqrt {1-x^2}} dx=\frac{c_k\pi }{2} \delta_{kj},
\end{equation}
where $c_0=2$ and $c_k=1$ for $k\ge 1.$  We have
\begin{align} \label{chebgretcff}
 T_k(x)=\frac 1 {2(k+1)} T_{k+1}'(x)-\frac 1 {2(k-1)}T_{k-1}'(x),\quad k\ge 2,
\end{align}
and
\begin{align}\label{chebpm1}
 T_k(\pm 1)=(\pm 1)^k,\quad T_k'(\pm 1)=(\pm 1)^{k-1} k^2.
\end{align}
The Chebyshev-Gauss-Lobatto (CGL) points and quadrature weights are
 \begin{equation}\label{CGLnodeweights}
  x_j=-\cos(jh),\;\; 0\le j\le N;\;\;  \omega_0=\omega_N=\frac h 2,\;\; \omega_j=h,\;\; 1\le j\le N-1;\;\;  h=\frac \pi N.
 \end{equation}
 Then we have the exactness
 \begin{equation}\label{Chebquadexac}
\int_{-1}^1 \frac{\phi(x)}{\sqrt{1-x^2}}\, dx=\frac{\pi} {2N}\big(\phi(-1)+\phi(1)\big)+\frac{\pi} N\sum_{j=1}^{N-1} \phi(x_j),\quad \forall \phi\in {\mathbb P}_{2N-1}.
\end{equation}

\subsection{Integration preconditioning} We briefly examine the essential idea of constructing integration preconditioners in  \cite{Hesthaven98,Elbarbary06} (inspired by \cite{CHT96,coutsias1996integration}).


We consider for example the Legendre case.  By \eqref{Legorth} and \eqref{quadexac},
  \begin{equation}\label{ljPk}
l_j(x)=\sum_{k=0}^N \frac{\omega_j}{\tilde \gamma_k} P_k(x_j)P_k(x), \quad 0\le j\le N,
\end{equation}
where $\tilde \gamma_k=2/(2k+1),$ for $0\le k\le N-1,$ and $\tilde \gamma_N=2/N$. Then 
\begin{equation}\label{ljPknewa}
l_j''(x)=\sum_{k=2}^N \frac{\omega_j}{\tilde \gamma_k} P_k(x_j)P_k''(x).
\end{equation}
 The  key observation in  \cite{Hesthaven98,Elbarbary06}  is that {\em pseudospectral differentiation
 process actually involves the ill-conditioned transform:}
 \begin{equation}\label{newmatrix}
{\rm span}\big\{P_k'' : 2\le k\le N\big\}:= Q_2^{N}\; \longmapsto\;  Q_0^{N-2}:={\rm span}\big\{P_k : 0\le k\le N-2 \big\}.
 \end{equation}
 Indeed, we have (see \cite[(3.176c)]{ShenTangWang2011}):
\begin{equation}\label{spdiff}
P_k''(x)=\sum_{k+l\; {\rm even}}^{0\le l\le k-2} (l+1/2)\big(k(k+1)-l(l+1)\big)P_l(x),
\end{equation}
so the transform matrix is dense and the coefficients grow like $k^2.$

However, {\em the inverse transform: $Q_0^{N-2} \mapsto  Q_2^{N}$ is sparse and well-conditioned},  thanks to the  ``compact" formula, derived from \eqref{legretcff}:
\begin{equation}\label{relationPddP}
P_k(x)=\alpha_k P''_{k-2}(x)+ \beta_k P''_{k}(x)+\alpha_{k+1}P''_{k+2}(x), \quad\;\;k\ge 2,
\end{equation}
where  the  coefficients are
\begin{equation}\label{coeffaa}
\alpha_k=\frac{1}{(2k-1)(2k+1)},\quad  \beta_{k}=-\frac{2}{(2k-1)(2k+3)},
\end{equation}
which decay like $k^{-2}.$

Based on \eqref{relationPddP}, \cite{Hesthaven98,Elbarbary06} attempted to precondition the collocation system by the ``inverse"  of  $\bs D^{(2)}.$ However,   since $\bs D^{(2)}$ is singular, there exist multiple ways to manipulate the involved singular matrices.   The boundary conditions  were imposed by the penalty method (cf. \cite{FG88}) in \cite{Hesthaven98}, and using auxiliary equations
in \cite{Elbarbary06}.  Note that the condition number of the preconditioned system  for e.g., the operator $\frac{d^2}{dx^2}-k$ with Dirichlet boundary conditions, behaves like  $O(\sqrt N).$

\subsection{Pseudospectral integration matrix}\label{outlins}
We take a quick glance at the idea of the new method in Section \ref{sect3}.
Slightly different from \eqref{diffpr},  we consider pseudospectral differentiation merely on interior GL points:
\begin{equation}\label{newdiffpr}
\widetilde{\bs D}^{(2)} \bs p=\tilde{\bs p}^{(2)}\;\; {\rm where}\;\; \tilde {\bs p}^{(2)}:=\big(p(-1),p^{(2)}(x_1)\cdots, p^{(2)}(x_{N-1}), p(1)\big)^t,
\end{equation}
and the matrix $\widetilde{\bs D}^{(2)}$ is obtained by replacing the first and last rows of
 $\bs D^{(2)}$  by the  row vectors $\bs e_1=(1,0,\cdots,0)$ and $\bs e_N=(0,\cdots,0,1),$ respectively.
 Note that the matrix $\widetilde{\bs D}^{(2)}$  is nonsingular.  More importantly, this also allows to impose boundary conditions exactly.

Based on Birkhoff interpolation, we obtain the exact inverse matrix, denoted by $\bs B,$ of  $\widetilde{\bs D}^{(2)}$ from the underlying Birkhoff interpolation basis.  Then we have the inverse process of \eqref{newdiffpr}:
\begin{equation}\label{newinteg}
\bs B\tilde{\bs p}^{(2)}=\bs p,
\end{equation}
which performs twice integration at the interior GL points, but remains the function values at endpoints
unchanged.  For this reason, we call $\bs B$ {\em the second-order pseudospectral integration matrix.}
It is important to point out that the computation of PSIM is stable  even for thousands of
collocation points, as all operations involve well-conditioned formulations (e.g., \eqref{relationPddP} is built-in).

\section{New collocation methods for second-order BVPs}\label{sect3}
\setcounter{equation}{0}
\setcounter{thm}{0}

In this section, we elaborate on the construction of the new approach outlined in Subsection \ref{outlins} in the context of
solving second-order BVPs.
We start with second-order BVPs with Dirichlet boundary conditions, and then consider general mixed boundary conditions
in late part of this section.

%

\subsection{Birkhoff interpolation at Gauss-Lobatto points}   Let $\{ x_j\}_{j=0}^N$
(with $x_0=-1$ and $x_N=1$) in \eqref{interpoint} be a set of  GL points. Consider the special case of \eqref{birkhoffprob}:
\begin{equation}\label{Birkhoffb}
\left\{
\begin{aligned}
& \text{Find $p\in {\mathbb P}_N$ such that  for any  $u\in C^2(I),$}\\
& p(-1)=u(-1);\;\; p''(x_j)=u''(x_j),\;\; 1\le j\le N-1;\;\; p(1)=u(1).
\end{aligned}\right.
\end{equation}
%
The Birkhoff interpolation polynomial $p$ of $u$ can be uniquely determined by
\begin{equation}\label{birkhoffint2}
p(x)=u(-1) B_0(x)+ \sum_{j=1}^{N-1} u''(x_j) B_j(x) + u(1) B_N(x),\quad x\in [-1,1],
\end{equation}
if one can find $\{B_j\}_{j=0}^N\subseteq {\mathbb P}_N,$ such that
\begin{align}
&B_0(-1)=1,\quad B_0(1)=0,\quad  B_0''(x_i)=0,\;\;\; 1\le i\le N-1; \label{H0basis1}\\
&B_j(-1)=0,\quad B_j(1)=0, \quad   B_j''(x_i)=\delta_{ij},\;\;\;  1\le i,j\le N-1; \label{H0basis2}\\
&B_N(-1)=0,\quad B_N(1)=1,\quad  B_N''(x_i)=0,\;\;\; 1\le i\le N-1.  \label{Hintbasis}
\end{align}
We call $\{B_j\}_{j=0}^N$ the {\em Birkhoff interpolation basis polynomials} of \eqref{Birkhoffb}, which are the counterpart of the Lagrange basis polynomials $\{l_j\}_{j=0}^N$.

  The basis $\{B_j\}_{j=0}^N$ can be uniquely expressed  by the following formulas.
\begin{theorem}\label{Birkhoffbasis2}
Let $\{x_j\}_{j=0}^N$ be a set of Gauss-Lobatto points. The Birkhoff  interpolation basis polynomials  $\{B_j\}_{j=0}^N$
defined in \eqref{H0basis1}-\eqref{Hintbasis}  are given by
\begin{align}
&B_0(x)=\frac{1 - x}{2},  \quad B_N(x)=\frac{1 + x}{2}; \label{interbasis1} \\
&B_j(x)=\frac{1+x} 2 \int_{-1}^1 (t-1) L_j(t)\, dt+\int_{-1}^x (x-t) L_j(t)\, dt, \;\;\; 1\le  j\le N-1, \label{interbasis2}
\end{align}
where  $\{L_j\}_{j=1}^{N-1}$ are the Lagrange basis polynomials {\rm (}of degree $N-2${\rm)} associated with $N-1$ interior Gauss-Lobatto points $\{x_j \}_{j=1}^{N-1},$ namely,
\begin{equation}\label{ljxexp2}
L_j(x)=\frac{Q_N(x)}{(x-x_j)Q_N'(x_j)},\quad Q_{N}(x)=\gamma_N  \prod_{j=1}^{N-1}(x-x_j),
\end{equation}
where $\gamma_N$ is any nonzero constant. Moreover, we have
\begin{equation}\label{interbasis1std}
B'_0(x)=-B'_N(x)=-\frac{1}{2};\;\;
B'_j(x)=\frac{1} 2\int_{-1}^1 (t-1) L_j(t)\, dt+\int_{-1}^x  L_j(t)\, dt, \;\;\; 1\le  j\le N-1.
\end{equation}
\end{theorem}
\begin{proof} One verifies readily from \eqref{H0basis1}-\eqref{H0basis2} that $B_0$ and $B_N$ must be linear polynomials given by \eqref{interbasis1}.
Using  \eqref{Hintbasis} and the fact $B_j''(x), L_j(x)\in {\mathbb P}_{N-2},$ we find that
$ B_j''(x)=L_j(x),$  so solving this ordinary differential equation with boundary conditions: $B_j(\pm1)=0,$ leads to the expression in  \eqref{interbasis2}. Finally, \eqref{interbasis1std} follows from \eqref{interbasis1}-\eqref{interbasis2}.
\end{proof}

Let $b_{ij}^{(k)}:=B_{j}^{(k)}(x_i),$ and define  the matrices
\begin{equation}\label{lobattocase}
\begin{split}
& \bs B^{(k)}=\big(b_{ij}^{(k)}\big)_{0\le i,j\le N}, \quad  \bs B^{(k)}_{\rm  in}=\big(b_{ij}^{(k)}\big)_{1\le i,j\le N-1},\quad k\ge 1.
\end{split}
 \end{equation}
 In particular, denote $b_{ij}:=B_{j}(x_i),$ $\bs B=\bs B^{(0)}$  and  $\bs B_{\rm in}=\bs B^{(0)}_{\rm in}.$
  \begin{rem}\label{BP2relation} The integration process \eqref{newinteg} is actually a direct consequence of \eqref{birkhoffint2}, as the Birkhoff interpolation polynomial of any $p\in {\mathbb P_N}$ is itself. \qed
\end{rem}


We have the following analogue  of \eqref{DkDk}, and this approach leads to the exact inverse of  second-order
PSDM associated with the interior interpolation points.
\begin{thm}\label{BkBkthm} There hold
 \begin{equation}\label{BJrela2}
\bs B^{(k)}=\bs D^{(k)}\bs B=\bs D^k \bs B=\bs D \bs B^{(k-1)},\quad k\ge 1,
\end{equation}
and
 \begin{equation}\label{BJrela200}
{\bs D}^{(2)}_{\rm in} \bs B_{\rm in}={\bs I}_{N-1},\quad \widetilde {\bs D}^{(2)}\bs B=\bs I_{N+1},
\end{equation}
where $\bs I_{M}$ is an  $M\times M$ identity matrix, and
the matrix $\widetilde {\bs D}^{(2)}$ is defined in \eqref{newdiffpr}.
\end{thm}
 \begin{proof} We first prove \eqref{BJrela2}. For any $\phi\in {\mathbb P}_N,$ we write $\phi(x)=\sum_{p=0}^N \phi(x_p) l_p(x),$ so we have
\begin{equation*}
 \phi^{(k)}(x)=\sum_{p=0}^N \phi(x_p) l_p^{(k)}(x),\quad k\ge 1.
 \end{equation*}
Taking $\phi=B_j (\in {\mathbb P}_N)$ and $x=x_i,$  we obtain
 \begin{equation}\label{dphik}
 b_{ij}^{(k)}=\sum_{p=0}^N d_{ip}^{(k)} b_{pj},\quad k\ge 1,
 \end{equation}
 which implies $\bs B^{(k)}=\bs D^{(k)}\bs B.$ The second equality follows from \eqref{DkDk}, and the last identity in
 \eqref{BJrela2} is due to the recursive relation $\bs B^{(k-1)}=\bs D^{k-1} \bs B$.

 We now turn to the proof of \eqref{BJrela200}.  It is clear that by  \eqref{H0basis2}, $b_{0j}=b_{Nj}=0$ for $1\le j\le N-1$ and
 $b_{ij}^{(2)}=\delta_{ij}$ for $1\le i,j\le N-1.$  Taking $k=2$ in \eqref{dphik} leads to
 $$
 \delta_{ij}=\sum_{p=1}^{N-1} d_{ip}^{(2)} b_{pj},\quad 1\le i,j\le N-1.
 $$
 This yields  ${\bs D}^{(2)}_{\rm in} \bs B_{\rm in}={\bs I}_{N-1},$ from which the second statement follows directly.
 \end{proof}

In view of Theorem \ref{BkBkthm}, we call $\bs B$ and $\bs B^{(1)}$ the second-order and first-order PSIMs, respectively.


\subsection{Computation of  PSIM}  
Now, we  present stable algorithms for computing the matrices $\bs B$ and $\bs B^{(1)}.$
Here, we just consider the Legendre and Chebyshev cases, but the method is extendable to general Jacobi polynomials straightforwardly.  For convenience,  we introduce the integral operators:
\begin{equation}\label{integral}
\partial_x^{-1} u(x)=\int_{-1}^x u(t)\,dt; \quad \partial_x^{-m}u(x)=\partial_x^{-1}\big(\partial_x^{-(m-1)}u(x)\big),\quad m\ge 2.
\end{equation}
By \eqref{legretcff}, \eqref{legpm1} and  \eqref{relationPddP}-\eqref{coeffaa},
\begin{align}  \label{Legenderivative}
& \partial^{-1}_xP_{k}(x)=\frac{1}{2k+1}\big(P_{k+1}(x)-P_{k-1}(x)\big),\;\; k\ge 1;\;\;\; \partial^{-1}_xP_{0}(x)=1+x,
\end{align}
and
\begin{equation}\label{Legen2edderivative}
\begin{split}
& \partial^{-2}_x P_{k}(x)=\frac{P_{k+2}(x)}{(2k+1)(2k+3)}-\frac{2P_{k}(x)}{(2k-1)(2k+3)}
+\frac{P_{k-2}(x)}{(2k-1)(2k+1)},\;\; k\ge 2;  \\
& \partial^{-2}_x P_0(x)=\frac{(1+x)^2} 2,\quad  \partial^{-2}_x P_1(x)=\frac{(1+x)^2(x-2)} 6.
\end{split}
\end{equation}
Similarly, we find from \eqref{chebgretcff} and  \eqref{chebpm1} that
\begin{equation}\label{chebyintegral}
\begin{split}
&  \partial_x^{-1} T_k(x)= \frac {T_{k+1}(x)} {2(k+1)}-\frac{T_{k-1}(x)}{2(k-1)}-\frac{(-1)^k}{k^2-1},\;\;\; k\ge 2; \\
& \partial_x^{-1} T_0(x)=1+x,\;\;  \partial_x^{-1} T_1(x)=\frac{x^2-1} 2.
\end{split}
\end{equation}
Using \eqref{chebyintegral} recursively yields
\begin{equation}\label{ddchebyintegral}
\begin{split}
 & \partial_x^{-2} T_k(x)= \frac {T_{k+2}(x)} {4(k+1)(k+2)}-\frac{T_{k}(x)}{2(k^2-1)}+
 \frac {T_{k-2}(x)} {4(k-1)(k-2)}-\frac{(-1)^k(1+x)}{k^2-1}\\
 &\qquad \qquad\;\; -\frac{3(-1)^k}{(k^2-1)(k^2-4)},\;\; k\ge 3;\\
& \partial^{-2}_x T_0(x)=\frac{(1+x)^2} 2,\;\;  \partial^{-2}_x T_1(x)
=\frac{(1+x)^2(x-2)} 6, \;\;  \partial^{-2}_x T_2(x)=\frac{x(1+x)^2(x-2)} 6.
\end{split}
\end{equation}
\begin{rem}\label{gjacobi}  Observe that $\partial^{-m}_xP_{k}(\pm 1)=0$ for all $k\ge m$ with $m=1,2,$
while $\partial^{-m}_xT_{k}(1)$  may not vanish.   The integrated Legendre and/or Chebyshev  polynomials are used to construct well-conditioned spectral-Galerkin methods,   $hp$ element methods (see  \cite{Shen94b,Shen02b,Guo.SW06}, and  \cite{Can09} for a review), and spectral integral methods (see e.g., \cite{clenshaw1957numerical,El69,Greengard91}). \qed
 \end{rem}

%
%
%

\begin{prop}[{\bf Birkhoff interpolation at LGL points}]\label{prop:LGL} Let $\{x_j,\omega_j\}_{j=0}^N$  be the LGL points and weights given in \eqref{LGLweights}.
Then the Birkhoff interpolation basis polynomials
  $\{B_j\}_{j=1}^{N-1}$  in Theorem  {\rm \ref{Birkhoffbasis2}}  can be computed by
  \begin{equation}\label{Bjbas2GGL}
  B_j(x)= \big(\beta_{1j}-\beta_{0j}\big) \frac{x+1}2 +\sum_{k = 0}^{N - 2}\beta_{kj}\frac{\partial_x^{-2}P_{k}(x)}{\gamma_k},
\end{equation}
where $\gamma_k=2/(2k+1),$ $\partial_x^{-2}P_{k}(x)$ is given in \eqref{Legen2edderivative}, and
\begin{equation}\label{Bjbas2LGLcoef}
\beta_{kj}=\bigg(P_k(x_j)-\frac{1-(-1)^{N+k}}{2} P_{N-1}(x_j)-\frac{1+(-1)^{N+k}}{2} P_N(x_j)\bigg)\omega_j.
\end{equation}
Moreover, we have
  \begin{equation}\label{dBjbas2GGL}
  B_j'(x)= \frac{\beta_{1j}-\beta_{0j}} 2 +\sum_{k = 0}^{N - 2}\beta_{kj}\frac{\partial_x^{-1}P_{k}(x)}{\gamma_k},
\end{equation}
where $\partial_x^{-1}P_{k}(x)$ is given in \eqref{Legenderivative}.
\end{prop}
\begin{proof}  Since $B_j''\in {\mathbb P_{N-2}},$ we expand it in terms of Legendre polynomials:
\begin{equation}\label{casecons}
B_j''(x)= \sum_{k = 0}^{N - 2}\beta_{kj}\frac{P_k(x)}{\gamma_k}\;\; {\rm where}\;\; \beta_{kj}=\int_{-1}^1 B_j''(x)P_k(x) dx.
\end{equation}
Using \eqref{quadexac}, \eqref{legpm1} and \eqref{H0basis2},  leads to
\begin{equation}\label{pftmp1}
{\beta_{kj}}=\int_{-1}^1 B_j''(x)P_k(x)dx =\big((-1)^kB_j''(-1) +B_j''(1)\big)\omega_0+P_k(x_j)\omega_j, \;\;  1\le j\le N-1.
\end{equation}
Notice that the last identity of \eqref{pftmp1} is valid for all $k\le N+1.$
Taking $k=N-1, N,$ we obtain from \eqref{Legorth} that the resulted integrals vanish, so we have the linear system of $B_j''(\pm 1)$:
\begin{equation*}
\begin{split}
& \big((-1)^{N-1}B_j''(-1) +B_j''(1)\big)\omega_0+P_{N-1}(x_j)\omega_j = 0,\\
& \big((-1)^{N}B_j''(-1) +B_j''(1)\big)\omega_0+ P_{N}(x_j)\omega_j= 0.
\end{split}
\end{equation*}
Therefore, we solve it and find that
\begin{equation}\label{Bjppvalue}
 B_j''(\pm 1) = -(\pm 1)^{N} \frac{\omega_j}{2\omega_0}\big(P_N(x_j)\pm  P_{N-1}(x_j)\big),\;\; 1\le j\le N-1.
\end{equation}
Inserting  \eqref{Bjppvalue} into \eqref{pftmp1} yields the expression for $\beta_{kj}$ in \eqref{Bjbas2LGLcoef}.
%
%

Next, it follows from \eqref{casecons} that
\begin{equation}\label{intbj}
B_j(x)= \sum_{k = 0}^{N - 2}\beta_{kj}\frac{\partial_{x}^{-2}P_k(x)}{\gamma_k}+C_1+C_2(x+1),
 \end{equation}
 where  $C_1$ and $C_2$ are constants to be  determined by $B_j(\pm 1)=0.$
Observe from \eqref{Legen2edderivative} that $\partial_{x}^{-2}P_k(-1)=0$ for $k\ge 0$ and
$\partial_{x}^{-2}P_k(1)=0$ for $k\ge 2.$  This implies $C_1=0$ and
$$
2C_2=-\frac{\beta_{0j}}{\gamma_0}\partial_{x}^{-2}P_0(1) -\frac{\beta_{1j}}{\gamma_1}\partial_{x}^{-2}P_1(1)=\beta_{1j}-\beta_{0j}.
$$
Thus,  \eqref{Bjbas2GGL} follows.  Finally, differentiating \eqref{Bjbas2GGL} leads to  \eqref{dBjbas2GGL}.
\end{proof}

  \begin{prop}[{\bf Birkhoff interpolation at CGL points}] \label{ChebyGL}
  The Birkhoff interpolation basis polynomials
  $\big\{B_j\big\}_{j=1}^{N-1}$ in Theorem {\rm \ref{Birkhoffbasis2}} at  CGL points
  $\big\{x_j=-\cos(jh)\big\}_{j=0}^N $  with $h=\pi/N,$ can be computed by
  \begin{equation}\label{Bjbass2Cheby}
 B_j(x) = \sum_{k = 0}^{N - 2}\beta_{kj}\Big\{ \partial^{-2}_xT_{k}(x)-\frac{1+x} 2 \partial_x^{-2} T_k(1)\Big\},
\end{equation}
where $\partial^{-2}_xT_{k}(x)$ is given in \eqref{ddchebyintegral}, and
\begin{equation}\label{Bjbas2Chebycoef}
\beta_{kj} =\frac{2}{c_kN}\bigg\{T_k(x_j)-\frac{1-(-1)^{N+k}}{2}T_{N-1}(x_j)-\frac{1+(-1)^{N+k}}{2}T_N(x_j) \bigg\}.
\end{equation}
Moreover, we have
  \begin{equation}\label{BjbassdCheby}
   B_j'(x) = \sum_{k = 0}^{N -2}\beta_{kj}\Big\{\partial^{-1}_xT_{k}(x)-\frac{\partial_x^{-2} T_k(1)}2\Big\},
\end{equation}
where $\partial^{-1}_xT_{k}(x)$ is computed by \eqref{chebyintegral}. Here, $c_0=2$ and $c_k=1$ for $k\ge 1$ as in
\eqref{Legorth}.
\end{prop}
Here, we omit  the proof, since it is very similar to that of  Proposition  \ref{prop:LGL}.

\begin{rem}\label{stableformula} Like  \eqref{relationPddP}-\eqref{coeffaa}, the formulas for evaluating
 integrated Legendre and/or Chebyshev  polynomials are sparse and the coefficients decay. This allows
 for stable computation of PSIM even for thousands of collocation points.   \qed
\end{rem}

 In Figure \ref{plots2}, we plot
the first six Birkhoff interpolation basis polynomials  at the GL points $\{x_j\}_{j=0}^5$ for both the Legendre (left) and Chebyshev
(right) cases.
\begin{figure}[!ht]
 \begin{center}
  \begin{minipage}{0.4\textwidth}
 \includegraphics[width=1\textwidth]{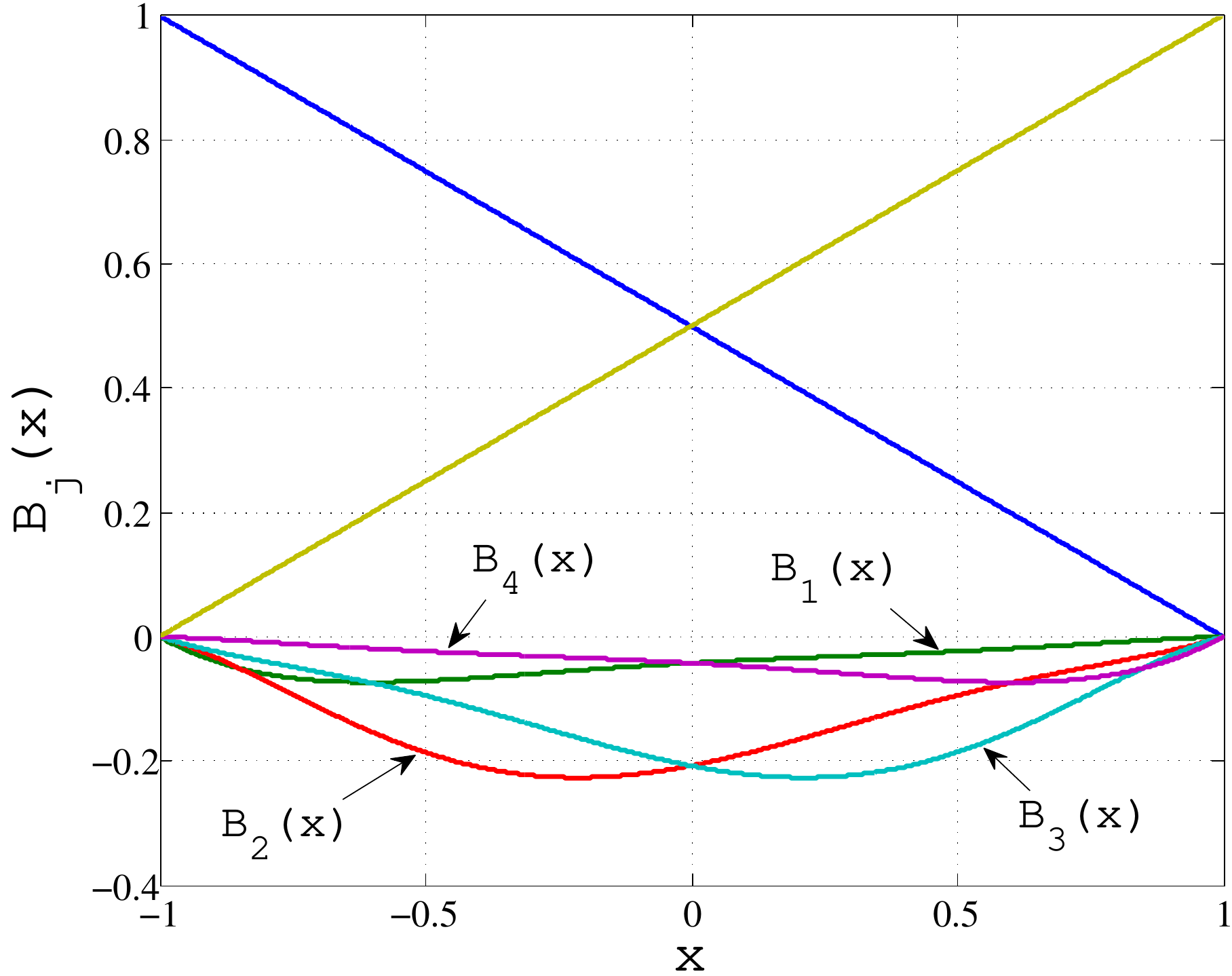}
  \end{minipage}\qquad
  \begin{minipage}{0.4\textwidth}
 \includegraphics[width=1\textwidth]{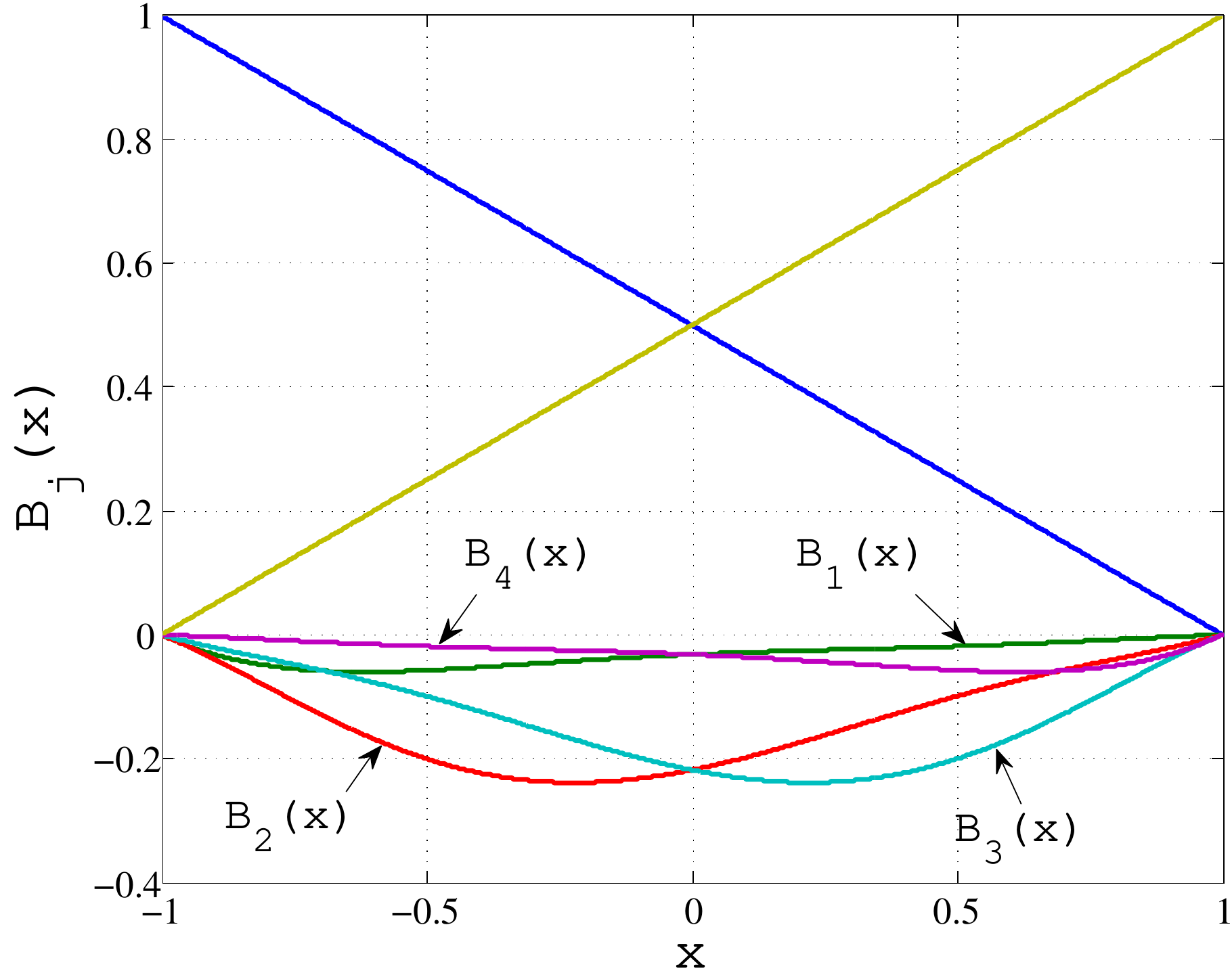}
  \end{minipage}
 \end{center}
 \caption{\small{Plots of $\{B_j\}_{j=0}^5$. Left: Legendre; right:  Chebyshev.}}
 \label{plots2}
 \end{figure}

\subsection{Collocation schemes}\label{Sect2sub2} Consider the BVP:
\begin{equation}\label{2nBVP}
\begin{split}
&  u''(x)+r(x) u'(x)+s(x) u(x)=f(x),\quad x\in I;\quad  u(\pm 1)=u_{\pm},
\end{split}
\end{equation}
where the given functions  $r,s,f\in C(I)$.   Let $\{x_j\}_{j=0}^{N}$ be the set of  Gauss-Lobatto points as in \eqref{Birkhoffb}. Then the collocation scheme for \eqref{2nBVP} is to find $u_N\in {\mathbb P}_N$ such that
\begin{equation}\label{2nBVPCol}
\begin{split}
&  u''_N(x_i)+r(x_i) u'_N(x_i)+s(x_i) u_N(x_i)=f(x_i),\quad 1\le i\le N-1;\quad  u_N(\pm 1)=u_{\pm}.
\end{split}
\end{equation}
As the Birkhoff interpolation polynomial of $u_N$ is itself, we have from \eqref{birkhoffint2} that
\begin{equation}\label{newsystem0}
u_N(x)=u_- B_0(x)+ u_+ B_N(x)+\sum_{j=1}^{N-1}u_N''(x_j) B_j(x).
\end{equation}
Then the matrix form of \eqref{2nBVPCol} reads
\begin{equation}\label{2nBVPColmat}
\big(\bs I_{N-1} +\bs \Lambda_r \bs B_{\rm in}^{(1)}+\bs \Lambda_s \bs B_{\rm in}\big)\bs v=\bs f-u_-\bs v_--u_+\bs v_+,
\end{equation}
where
\begin{align*}
& \bs \Lambda_r={\rm diag}\big(r(x_1),\cdots,r(x_{N-1})\big), \;\; \bs \Lambda_s={\rm diag}\big(s(x_1),\cdots,s(x_{N-1})\big),\\
& \bs v=\big(u_N''(x_1),\cdots,u_N''(x_{N-1})\big)^t,\quad \bs f=\big(f(x_1),\cdots, f(x_{N-1})\big)^t,\\
& \bs v_-=\Big(-\frac {r(x_1)} 2+s(x_1) \frac{1-x_1} 2,\cdots,  -\frac {r(x_{N-1})} 2+s(x_{N-1}) \frac{1-x_{N-1}} 2 \Big)^t,\\
& \bs v_+=\Big(\frac {r(x_1)} 2+s(x_1) \frac{1+x_1} 2,\cdots,  \frac {r(x_{N-1})} 2+s(x_{N-1}) \frac{1+x_{N-1}} 2 \Big)^t.
\end{align*}
It is seen that under the new basis  $\{B_j\},$  the  matrix of the highest derivative is identity, and it also allows for exact imposition of  boundary conditions.

 In summary, we take the following steps to solve \eqref{2nBVPCol}:
\begin{itemize}
\item Pre-compute $\bs B$ and $\bs B^{(1)}$ via the formulas in Propositions \ref{prop:LGL}-\ref{ChebyGL};
\item Find $\bs v$ by solving the system \eqref{2nBVPColmat};
\item Recover $\bs u=(u_N(x_1),\cdots,u_N(x_{N-1}))^t$ from   \eqref{newsystem0}:
\begin{equation}\label{urecovery}
\bs u= \bs B_{\rm in} \bs v+u_- \bs b_0+u_+\bs b_N,
\end{equation}
where $\bs b_j=\big(B_j (x_1),\cdots, B_j(x_{N-1})\big)^t$ for $ j=0,N.$
\end{itemize}

For comparison, we look at  the usual collocation scheme \eqref{2nBVPCol} under the Lagrange basis.
Write 
$$
u_N(x)=u_- l_0(x)+u_+ l_N(x)+\sum_{j=1}^{N-1} u_N(x_j) l_j(x),
$$
and insert it into  \eqref{2nBVPCol},  leading to  
\begin{equation}\label{ususcoll}
\big(\bs D_{\rm in}^{(2)}+\bs \Lambda_r \bs D^{(1)}_{\rm in} +\bs \Lambda_s  \big)\bs u=\bs f-\bs u_{_B},
\end{equation}
where  $\bs f$ is the same as in \eqref{2nBVPColmat},
$\bs u$ is the vector of unknowns $\{u_N(x_i)\}_{i=1}^{N-1},$ and  $\bs u_{_B}$ is the vector of
$\big\{u_-(d_{i0}^{(2)}+r(x_i) d_{i0}^{(1)})+u_+(d_{iN}^{(2)}+r(x_i) d_{iN}^{(1)})\big\}_{i=1}^{N-1}.$
It is known that the condition number of the coefficient matrix in \eqref{ususcoll} grows like $O(N^4).$

Thanks to the property:  $\bs B_{\rm in} \bs D_{\rm in}^{(2)}={\bs I}_{N-1}$ (see Theorem \ref{BkBkthm}),  the matrix $\bs B_{\rm in}$ can be used to precondition the ill-conditioned system \eqref{ususcoll}, leading to
\begin{equation}\label{ususcoll3}
\big(\bs I_{N-1}+\bs B_{\rm in}\bs \Lambda_r \bs D^{(1)}_{\rm in} +\bs B_{\rm in} \bs \Lambda_s  \big)\bs u=\bs B_{\rm in}\big(\bs f-\bs u_{_B}\big).
\end{equation}
\begin{rem}\label{optcond}  Different from  \cite{Hesthaven98,Elbarbary06}, we work with the system
 involving $\bs D^{(2)}_{\rm in}$ (i.e., unknowns at interior points), rather than  $\bs D^{(2)}.$ Moreover,
 the boundary conditions are imposed exactly (see Subsection \ref{sect:mixedBC} for general mixed boundary conditions), rather than using the penalty method \cite{Hesthaven98} and auxiliary equations \cite{Elbarbary06}. Consequently, our approach leads to optimal IPs and well-conditioned preconditioned systems.  \qed
\end{rem}

We now make a comparison of condition numbers between  the above linear systems and
 IP  in \cite{Elbarbary06}. Consider the same example as in \cite[Section 7]{Elbarbary06}:
\begin{equation}\label{prob2v}
    u''(x) - xu'(x) - u(x) = 0,\quad x\in I; \quad u(\pm 1) = 1,
\end{equation}
with the exact solution $u(x) = e^{(x^2 - 1)/2}.$
In Table \ref{GLtab2}, we tabulate the condition numbers (``Cond.$\#$") and maximum pointwise errors between the numerical and exact solutions obtained from
the Lagrange collocation (LCOL) scheme \eqref{ususcoll}, the Birkhoff collocation (BCOL) scheme
\eqref{2nBVPColmat} and the preconditioned LCOL (P-LCOL) scheme (\ref{ususcoll3}), respectively.
We also compare with \cite[Tables~2--3]{Elbarbary06}.  Observe that the condition numbers of the new approaches are independent of $N,$ and
do not induce round-off errors.

{\small
\begin{table}[!th]
\caption{\small Comparison of results with \cite[Tables~2--3]{Elbarbary06}}
\vspace*{-6pt}
\begin{tabular}{|c|c|c|c|c|c|c|c|c|}
    \hline
    \multirow{3}{*}{$N$} & \multicolumn{2}{|c|}{LCOL (\ref{ususcoll})} & \multicolumn{2}{|c|}{Results from \cite{Elbarbary06}}
     & \multicolumn{2}{|c|}{BCOL (\ref{2nBVPColmat})} & \multicolumn{2}{|c|}{P-LCOL (\ref{ususcoll3})}\\
    \cline{2-9}
    & Cond.$\#$ & Error & Cond.$\#$ & Error & Cond.$\#$ & Error & Cond.$\#$ & Error\\
        \hline
             \multicolumn{9}{|c|}{Legendre}\\
    \hline
64 & 1.51e+05 & 1.65e-13 & 37.2 & 9.99e-16 & 1.90 & 5.55e-16 & 1.32 & 1.22e-15\\
    \hline
128 & 2.37e+06 & 5.46e-13 & 75.5 & 1.33e-15 & 1.92 & 6.66e-16 & 1.32 & 1.44e-15\\
    \hline
256 & 3.76e+07 & 1.40e-12 & 146 & 2.55e-15 & 1.93 & 1.11e-15 & 1.32 & 2.00e-15\\
    \hline
512 & 5.99e+08 & 1.96e-11 & 292 & 3.11e-15 & 1.93 & 1.89e-15 & 1.32 & 3.11e-15\\
    \hline
1024 & 7.21e+09 & 3.21e-11 & 582 & 6.81e-15 & 1.94 & 3.22e-15 & 1.32 & 5.77e-15\\
    \hline
        \multicolumn{9}{|c|}{Chebyshev}\\
    \hline
64 & 2.74e+05 & 7.14e-14 & 37.3 & 9.99e-16 & 1.91 & 7.77e-16 & 1.32 & 9.99e-16\\
    \hline
128 & 4.39e+06 & 5.74e-13 & 73.7 & 1.78e-15 & 1.93 & 7.77e-16 & 1.32 & 1.22e-15\\
    \hline
256 & 7.02e+07 & 2.22e-12 & 146 & 2.99e-15 & 1.93 & 1.22e-15 & 1.32 & 1.89e-15\\
    \hline
512 & 1.12e+09 & 9.52e-12 & 292 & 3.89e-15 & 1.94 & 1.67e-15 & 1.32 & 2.66e-15\\
    \hline
1024 & 1.80e+10 & 4.61e-11 & 583 & 7.44e-15 & 1.94 & 3.77e-15 & 1.32 & 4.77e-15\\
    \hline
\end{tabular}\label{GLtab2}
\end{table}
}

As a second example, we consider
\begin{equation}\label{prob3r}
  u''(x) - u(x) =f(x)=
  \begin{cases}
  	\frac{x^2}{2} + x - 1,\;\; & -1 < x < 0,\\
  	x - 1, & 0 \le x < 1,
  \end{cases}
\end{equation}
with the exact solution
\begin{equation*}\label{prob3rsoln}
	u(x) = \begin{cases}
		\cosh(x + 1) - \frac{x^2}{2} - x,\;\; &  -1 \le x < 0,\\
		\cosh(x + 1) - \cosh(x) - x + 1,\;\; &  0 \le x \le 1.
	\end{cases}
\end{equation*}
Note that $f\in C^1(\bar I)$  and  $u\in C^3(\bar I).$
In Figure~\ref{accuracygraphs}, we graph the maximum point-wise  errors for both BCOL and LCOL.
We see that the BCOL is free of round-off error even for thousands of points. Note that the slope of the line is approximately $-3$ as expected.
 \begin{figure}[!ht]
        \begin{center}
            \begin{minipage}{0.4\textwidth}
                \includegraphics[width=1.0\textwidth]{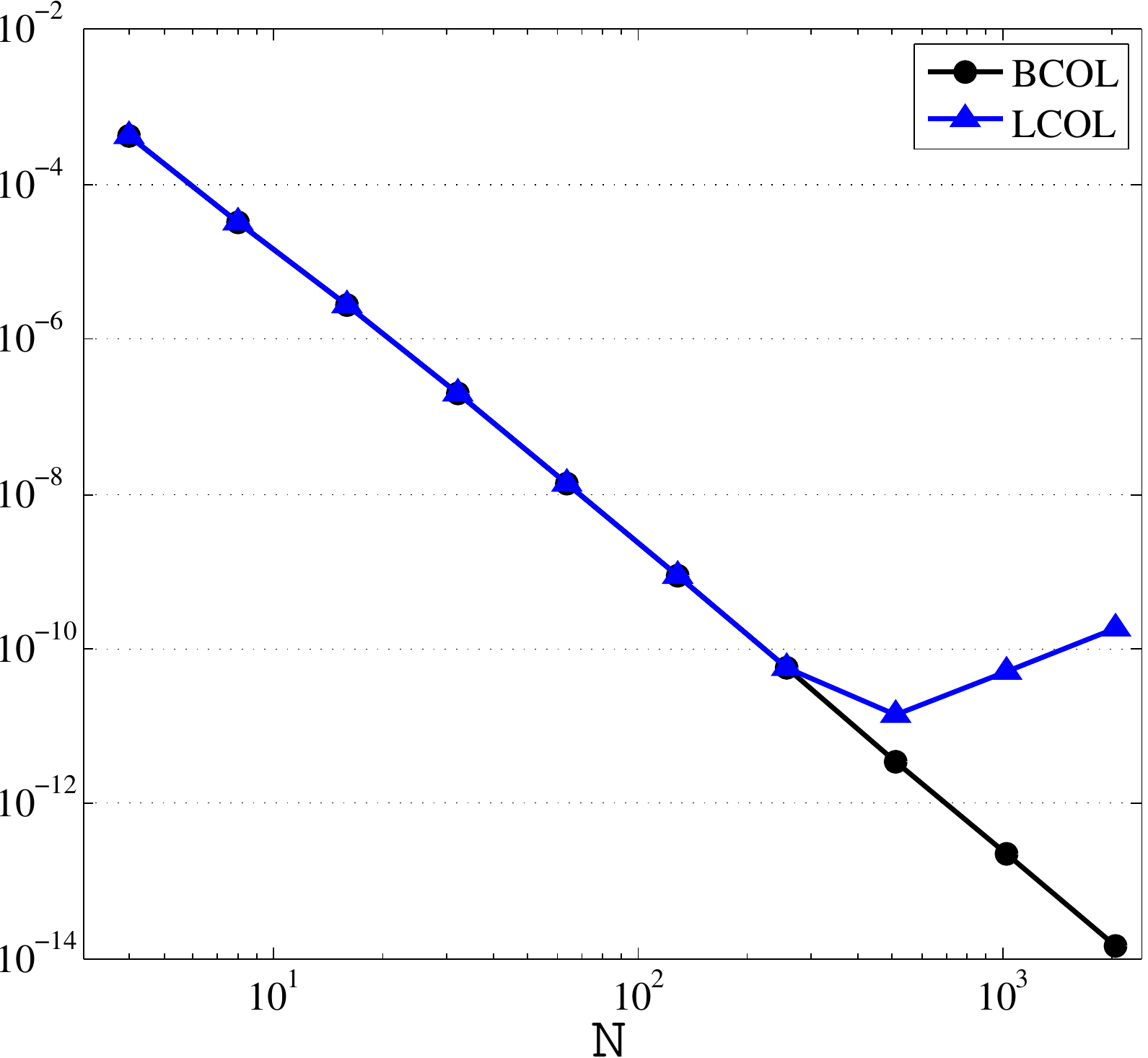}
            \end{minipage}\qquad
            \begin{minipage}{0.4\textwidth}
                \includegraphics[width=1.0\textwidth]{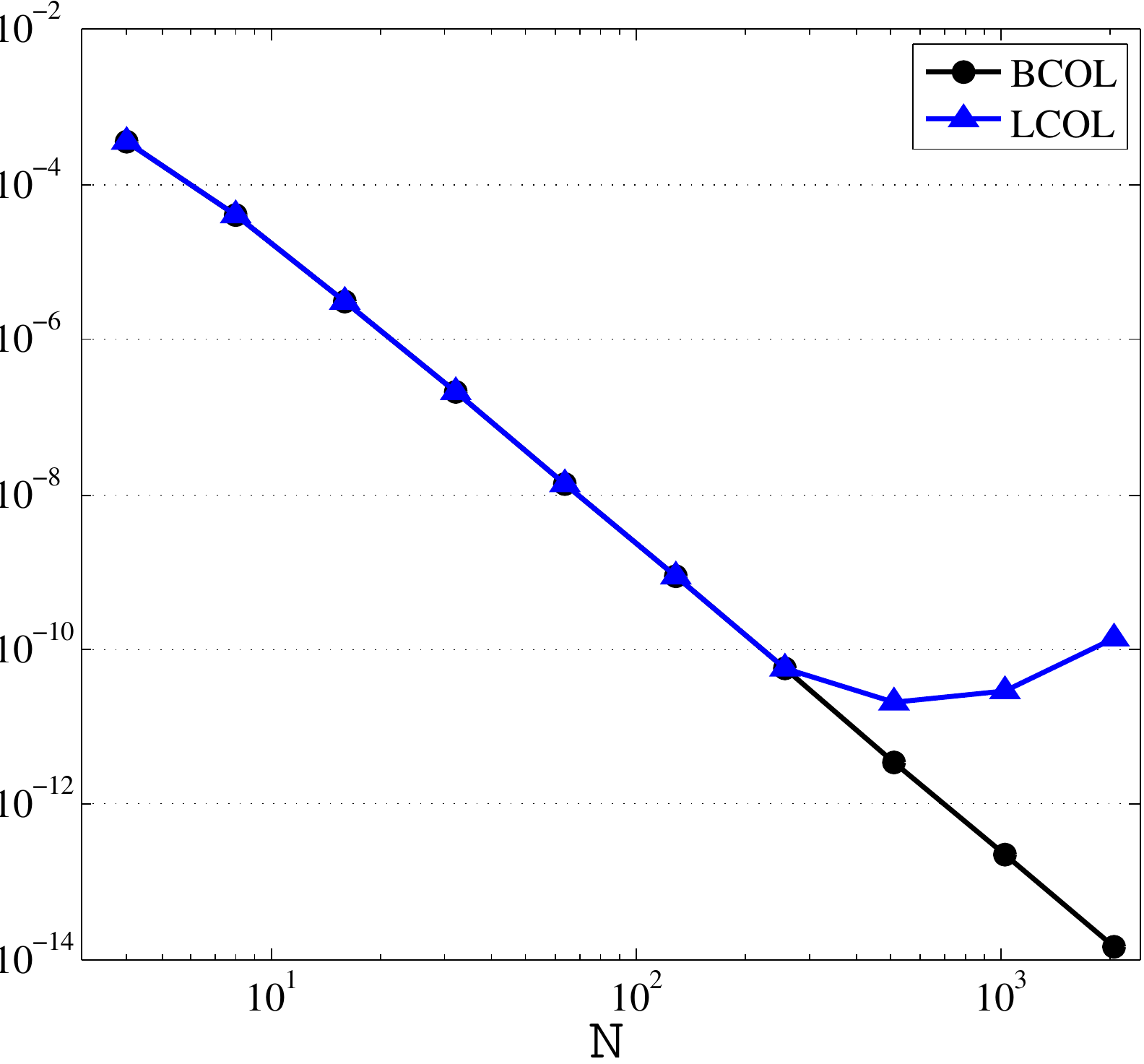}
            \end{minipage}
        \end{center}
        \caption{\small{Comparison of maximum pointwise errors. Left: LGL; right: CGL.}}
        \label{accuracygraphs}
    \end{figure}


Below, we have some insights into eigenvalues of the new collocation system  for
 the operator: $\frac{d^2}{{dx^2}}-k$ (i.e.,  Helmholtz  (resp. modified Helmholtz) operator for $k<0$ (resp. $k>0$)) with Dirichlet boundary conditions.
\begin{prop}\label{addprof}  In the LGL case,  the eigenvalues of $\bs I_{N-1}-k \bs B_{\rm in}$ are all real and distinct, which are uniformly bounded. More precisely, for any eigenvalue $\lambda$ of $\bs I_{N-1}-k \bs B_{\rm in},$ we have
\begin{equation}\label{eignbound}
1+c_N\frac{4k\pi^2}{N^4}<\lambda< 1+\frac{4k}{\pi^2},\;\; {\rm if}\;\; k\ge 0; \quad  1+\frac{4k}{\pi^2}<\lambda< 1+c_N\frac{4k\pi^2}{N^4},\;\; {\rm if}\;\; k<0,
\end{equation}
where  $c_N\approx 1$ for large $N.$
\end{prop}
\begin{proof} From  \cite[Theorem 7]{Welfert1994}, we know that all eigenvalues of $\bs D_{\rm in}^{(2)},$ denoted by
 $\{\lambda_{N,l}\}_{l=1}^{N-1},$ are real, distinct and negative, which we arrange them as $\lambda_{N,N-1}<\cdots<\lambda_{N,1}<0.$
We diagonalize $\bs D_{\rm in}^{(2)}$ and write it  as $ \bs D_{\rm in}^{(2)}=\bs Q \bs \Lambda_\lambda \bs Q^{-1},$ where
$\bs Q$ is formed by the eigenvectors and $\bs \Lambda_\lambda$ is the diagonal matrix of all eigenvalues. Since
$\bs B_{\rm in}=\big(\bs D_{\rm in}^{(2)}\big)^{-1}$ (cf.  Theorem
\ref{BkBkthm}), we have $\bs I_{N-1}-k \bs B_{\rm in}=\bs Q \bs (\bs I_{N-1}- k\bs \Lambda_\lambda^{-1}\big)\bs Q^{-1}.$ Therefore, the eigenvalues of $\bs I_{N-1}-k \bs B_{\rm in}$ are $\{1-k\lambda_{N,l}^{-1}\}_{l=1}^{N-1},$ which are real and distinct.
Then the bounds in \eqref{eignbound} can be obtained from the properties:
$ -\lambda_{N,1} > \pi^2/4$ (see
   \cite[Last line on Page 286]{Welfert1994}  and  \cite[Theorem 2.1]{Boulmezaoud2007}), and
   $-\lambda_{N,N-1}=c_N N^4/(4\pi^2)$ (see \cite[Proposition 9]{Welfert1994}).  
\end{proof}

\begin{rem}\label{Chebycond} We can obtain similar bounds for the CGL case by using the bounds for eigenvalues of $\bs D_{\rm in}^{(2)}$ in  e.g., \cite{Weideman1988} and  \cite[Section 4.3]{CHQZ06}. \qed
\end{rem}

\begin{rem}\label{condLGL}  As a consequence of \eqref{eignbound},    the condition number of $\bs I_{N-1}-k \bs B_{\rm in}$
is independent of $N.$ For example, it is uniformly bounded by $1+ {4k}/{\pi^2}$ for $k\ge 0.$  It is noteworthy that if $k=-w^2$ with $w\gg 1$ (i.e., Helmholtz equation with high wave-number), then the condition number behaves like $O(w^2)$, independent of $N$. \qed
\end{rem}

\subsection{Mixed boundary conditions}\label{sect:mixedBC}
Consider the second-order BVP \eqref{2nBVP},  equipped with mixed boundary conditions:
\begin{equation}\label{genbdy}
\begin{split}
    {\mathcal B}_-[u]:= a_-u(-1) + b_-u'(-1) = c_-,\quad
{\mathcal B}_+[u]:= a_+u(1) + b_+u'(1) = c_+,
\end{split}
\end{equation}
where $a_\pm, b_\pm$ and $c_\pm$ are given constants.  We first  assume  that
\begin{equation}\label{genbdydena}
    d := 2a_+a_- - a_+b_- + a_-b_+ \not =0,
\end{equation}
which excludes Neumann boundary conditions (i.e., $a_-=a_+=0$) to be considered later.

We associate  \eqref{genbdy}  with  the Birkhoff-type interpolation:
\begin{equation}\label{birkhoffprobmixed}
\left\{\begin{aligned}
&\text{Find $p\in {\mathbb P}_N$  such that}\\
& {\mathcal B}_-[p]=c_-,\;\; p''(x_j)=c_j,\;\; 1\le j\le N-1,\;\;  {\mathcal B}_+[p]=c_+,
\end{aligned}\right.
\end{equation}
where $\{x_j\}$ are  interior Gauss-Lobatto points, and $\{c_\pm, c_j\}$ are  given.
As before, we look for the interpolation basis polynomials, still denoted by $\{B_j\}_{j=0}^N,$ satisfying
\begin{equation}\label{genbdybasis2}
\begin{split}
    &{\mathcal B}_- [B_0] = 1,  \quad B_0''(x_i) = 0, \;\; 1\le i\le N-1, \quad  {\mathcal B}_+ [B_0] = 0;\\
    & {\mathcal B}_- [B_j] = 0,  \quad B_j''(x_i) = \delta_{ij},\;\; 1\le i\le N-1, \quad {\mathcal B}_+[B_j] = 0,\;\;  1\le j\le N-1;\\
    & {\mathcal B}_-  [B_N] = 0,  \quad B_N''(x_i) = 0,\;\; 1\le i\le N-1,\quad   {\mathcal B}_+ [B_N] = 1.
\end{split}
\end{equation}
Following the same lines as for the proof of Theorem  \ref{Birkhoffbasis2}, we find that if $d\not =0,$
\begin{equation}\label{genbdyvermodcc}
        B_0(x) = \frac{a_+} d (1 - x) + \frac {b_+}{d},\quad  B_N(x) = \frac{a_-} d (1 + x) - \frac {b_-}{d},
\end{equation}
and for $1\le j\le N-1,$
\begin{equation}\label{Bjformmixed}
\begin{split}
        B_j(x) & = \int^x_{-1}(x - t)L_j(t)\,{d}t - \Big(\frac{a_-} d (1 + x) - \frac {b_-}{d}\Big)
        \int_{-1}^1 \big(a_+(1-t)+b_+\big) L_j(t)\, dt,
\end{split}
\end{equation}
where $\{L_j\}$ are the Lagrange basis polynomials associated with the interior Gauss-Lobatto points as defined in
Theorem  \ref{Birkhoffbasis2}.
Thus, for any  $u\in C^2(I),$ its interpolation polynomial is given by
\begin{equation}\label{interpBirkh}
p(x)=\big({\mathcal B}_-[u]\big) B_0(x)+\sum_{j=1}^{N-1} u''(x_j) B_j(x)+\big({\mathcal B}_+[u]\big) B_N(x).
\end{equation}
We can find  formulas for computing $\{B_j\}_{j=1}^{N-1}$ on LGL and CGL points by using the same approach as in Proposition \ref{prop:LGL}.

 Armed with the new basis, we can impose mixed boundary conditions exactly, and the
linear  system resulted from the usual collocation scheme  is well-conditioned.
 Here, we test the method on the second-order equation in \eqref{2nBVP} but with the mixed boundary conditions: $u(\pm 1)\pm u'(\pm 1)=u_\pm.$
In Table \ref{GLtabo}, we list the condition numbers of the usual collocation method (LCOL, where the boundary conditions are treated by the tau-method), and the Birkhoff collocation method (BCOL) for both Legendre and Chebyshev cases.  Once again,
the new approach is well-conditioned.
%
%
{\small
\begin{table}[!th]
\caption{Comparison of  condition numbers}
\vspace*{-6pt}
\begin{tabular}{|c|c|c|c|c|c|
c|c|c|}
    \hline
    \multirow{3}{*}{$N$} & \multicolumn{4}{|c|}{$r=0$ and $s=-1$} & \multicolumn{4}{|c|}{$r=s=-1$}\\
    \cline{2-9}
    & \multicolumn{2}{|c|}{Chebyshev} & \multicolumn{2}{|c|}{Legendre} & \multicolumn{2}{|c|}{Chebyshev} & \multicolumn{2}{|c|}{Legendre}\\
    \cline{2-9}
    & BCOL & LCOL & BCOL & LCOL & BCOL & LCOL & BCOL & LCOL\\
    \hline
    32 & 2.42 & 1.21e+05 & 2.45 & 6.66e+04 & 2.61 & 1.43e+05 & 2.61 & 7.87e+04\\
    \hline
    64 & 2.43 & 2.65e+06 & 2.45 & 1.41e+06 & 2.63 & 3.15e+06 & 2.63 & 1.68e+06\\
    \hline
    128 & 2.44 & 5.88e+07 & 2.45 & 3.09e+07 & 2.64 & 7.04e+07 & 2.64 & 3.70e+07\\
    \hline
    256 & 2.44 & 1.32e+09 & 2.45 & 6.88e+08 & 2.64 & 1.58e+09 & 2.64 & 8.26e+08\\
    \hline
    512 & 2.44 & 2.97e+10 & 2.44 & 1.54e+10 & 2.65 & 3.57e+10 & 2.65 & 1.86e+10\\
    \hline
    1024 & 2.44 & 6.71e+11 & 2.44 & 3.48e+11 & 2.65 & 8.08e+11 & 2.65 & 4.19e+11\\
    \hline
\end{tabular}\label{GLtabo}
\end{table}}

\subsection{Neumann boundary conditions}  Consider  the Poisson equation with Neumann boundary conditions:
\begin{equation}\label{2ndNeumann}
u''(x)=f(x),\quad  x\in I;\quad u'(\pm 1)=0,
\end{equation}
where $f$ is a continuous function such that $\int_{-1}^1 f(x)\,dx=0.$  Its solution is unique up to any additive  constant.
To ensure the uniqueness, we supply \eqref{2ndNeumann} with an additional condition: $u(-1)=u_-.$

Observe that the interpolation problem  \eqref{birkhoffprobmixed} is not well-posed if ${\mathcal B}_\pm [u]$ reduces to Neumann boundary conditions. Here, we consider the following special case of  \eqref{birkhoffprob}:
\begin{equation}\label{birkhoffprobneumann}
\begin{cases}
\text{Find } p\in {\mathbb P}_{N + 1} \text{ such that}\\
p(-1)=y_0^{0},\;\; p'(-1)=y_0^1,\;\; p''(x_j)=y_j^2,\;\; 1\le j\le N-1,\;\;  p'(1) = y_N^1,
\end{cases}
\end{equation}
where $\{x_j\}_{j=1}^{N-1}$ are interior Gauss-Lobatto points, and the data $\{y_j^m\}$ are  given.
However, this interpolation problem is only conditionally well-posed.
For example, in the LGL and CGL cases, we have to assume that $N$ is odd.

 As before, we look for  basis polynomials, still denoted by $\{B_j\}_{j=0}^{N+1},$ such that for
 $1\le i\le N-1,$
\begin{equation}\label{neumannbdybasis2}
\begin{split}
    & B_0(-1) = 0, \;\; B_0'(-1) = 1,  \;\; B_0''(x_i) = 0, \;\;   B_0'(1) = 0;\\
    & B_j(-1) = 0, \;\; B_j'(-1) = 0,  \;\; B_j''(x_i) = \delta_{ij}, \;\; B_j'(1) = 0,\;\;  1\le j\le N-1;\\
    & B_N(-1) = 0, \;\; B_N'(-1) = 0,  \;\; B_N''(x_i) = 0,\;\;  B_N'(1) = 1;\\
    & B_{N+1}(-1) = 1, \;\; B_{N+1}'(-1) = 0,  \;\; B_{N+1}''(x_i) = 0, \;\;   B_{N+1}'(1) = 0.
\end{split}
\end{equation}
Let $Q_{N}(x)=c_N\prod_{j=1}^{N-1} (x-x_j)$ with $c_N\not=0$ as defined in \eqref{ljxexp2}.
Following the proof of Theorem \ref{Birkhoffbasis2}, we find that if
$\int_{-1}^1 Q_N(t)\,dt\not=0,$ we have
\begin{equation}\label{neumannbdyvermodcc}
        B_0(x) = 1 + x - \frac{\int_{-1}^x (x - t)Q_N(t)\,dt}{\int_{-1}^1 Q_N(t)\,dt},\;\;  B_N(x) = \frac{\int_{-1}^x (x - t)Q_N(t)\,dt}{\int_{-1}^1 Q_N(t)\,dt},\;\;  B_{N + 1}(x) \equiv 1,
\end{equation}
and for $1\le j\le N-1,$
\begin{equation}\label{Bjformneumann}
        B_j(x) = \int^x_{-1}(x - t)L_j(t)\,dt - 
       \bigg(\int_{-1}^1 L_j(t)\,dt\bigg)B_N(x),\;\;\; L_j(x)=\frac{Q_N(x)}{(x-x_j)Q_{N}'(x_j)}.
\end{equation}
\begin{rem}\label{Legcheb}
In the Legendre/Chebyshev case, we have $Q_N(x)=P_N'(x)$ or $T_N'(x),$ so by \eqref{legpm1}-\eqref{chebpm1},
$$
\int_{-1}^1 Q_N(t)\,dt=\int_{-1}^1 P'_N(t)\,dt=1-(-1)^N=\int_{-1}^1 T'_N(t)\,dt,
$$
which is nonzero, if and only if $N$ is odd. \qed
\end{rem}

%
%
 We plot in Figure \ref{accuracygraphn} the maximum point-wise errors of  the usual collocation (LCOL) and Birkhoff collocation (BCOL) methods for \eqref{2ndNeumann} with the exact solution $u(x) = \cos(10 x) - \cos(10).$ Note that the condition numbers of systems obtained from BCOL are all $1.$  We see that BCOL outperforms LCOL as before.

\begin{figure}[!ht]
        \begin{center}
            \begin{minipage}{0.4\textwidth}
                \includegraphics[width=1\textwidth]{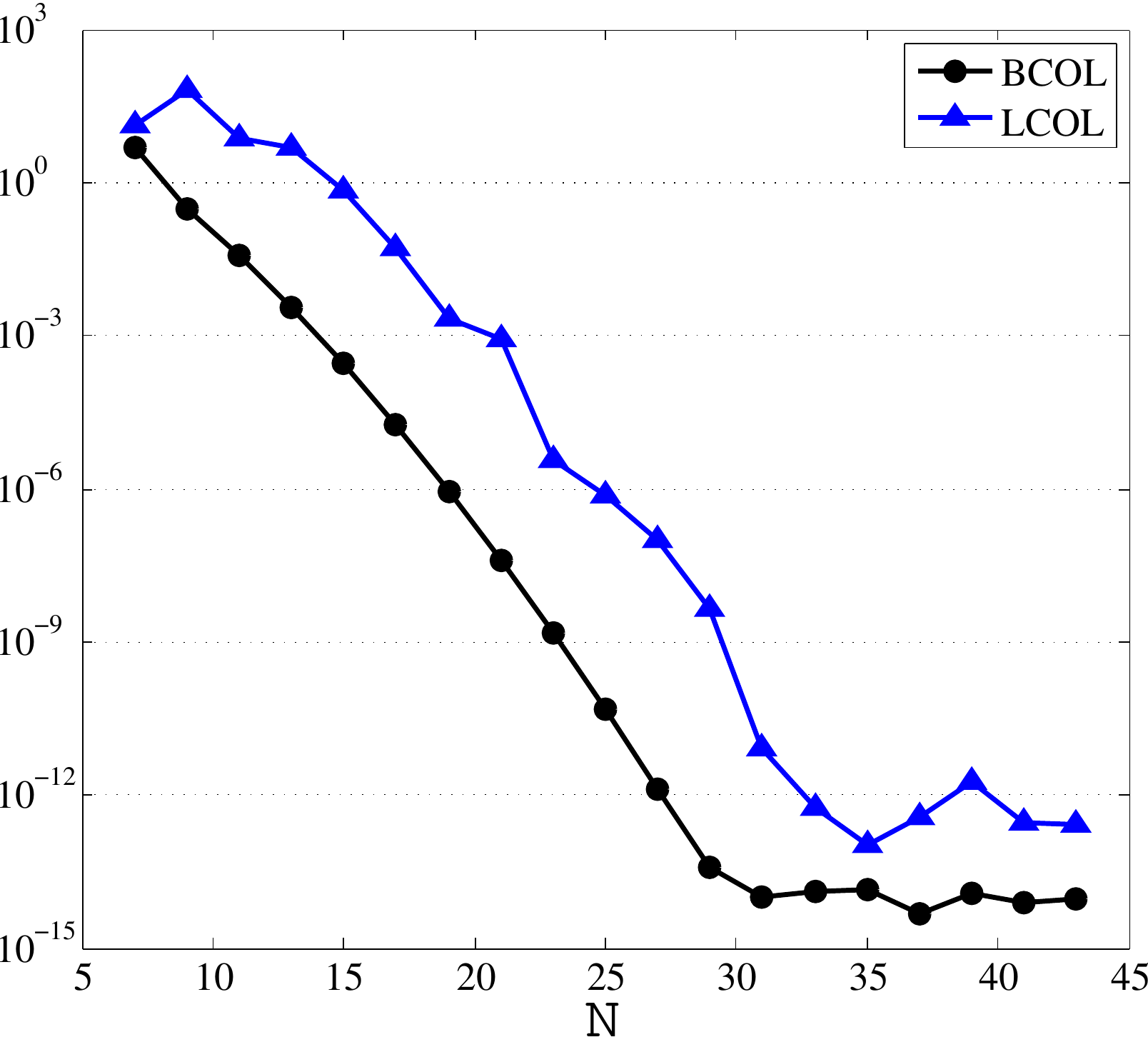}
            \end{minipage}\qquad
            \begin{minipage}{0.4\textwidth}
                \includegraphics[width=1\textwidth]{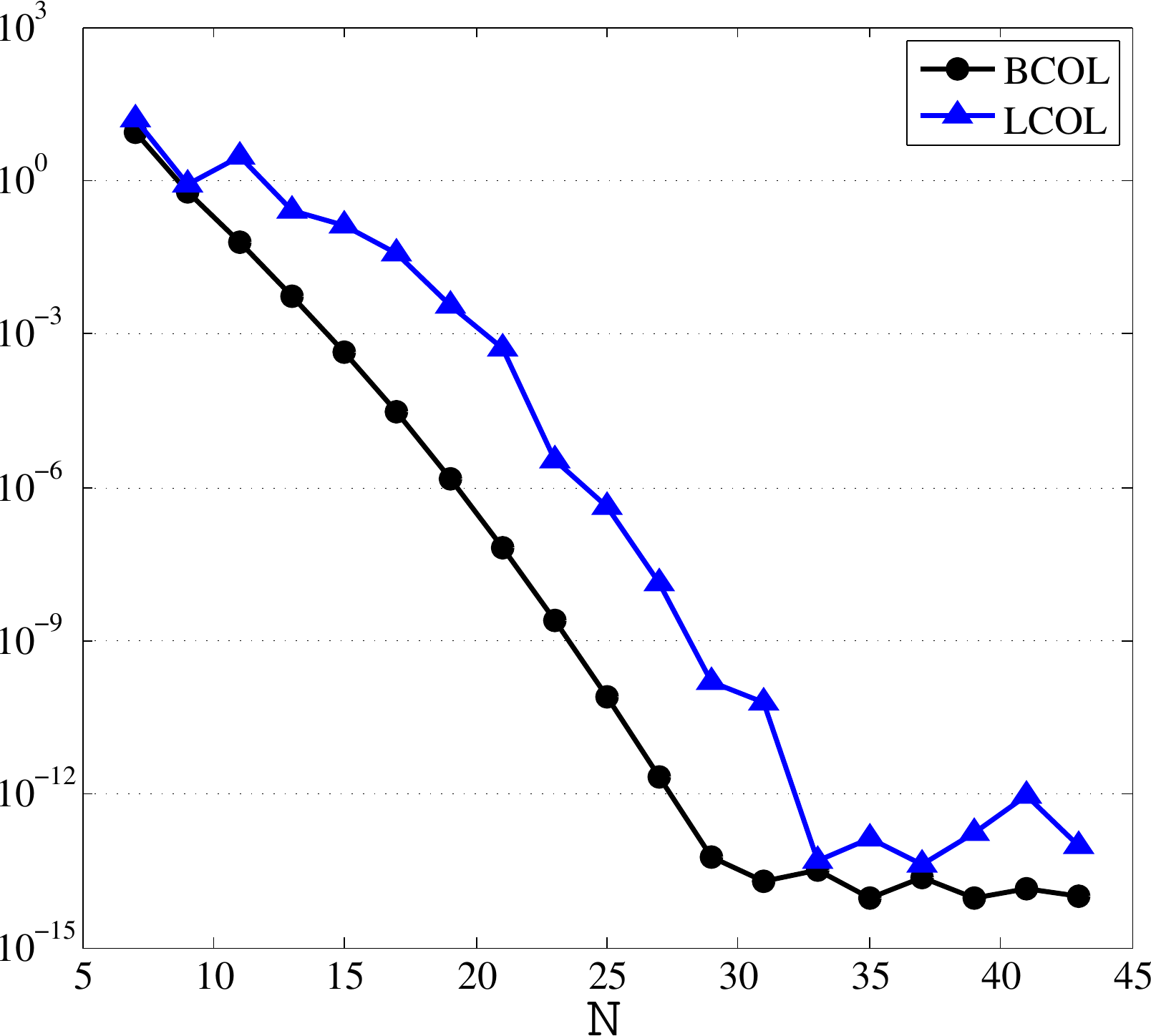}
            \end{minipage}
        \end{center}
        \caption{\small{Comparison of maximum pointwise errors. Left: LGL; right: CGL.}}
        \label{accuracygraphn}
    \end{figure}

\section{Miscellaneous extensions and discussions}\label{sect4}

In this section, we present various extensions of the Birkhoff interpolation and new collocation methods to numerical
solution of first-order initial value problems (IVPs),  higher order equations, and multi-dimensional problems.

\subsection{First-order IVPs}
To this end, let  $\{ x_j\}_{j=0}^N$ in \eqref{interpoint} be a set of Gauss-Radau  interpolation points (with $x_0=-1$ and $x_N<1$).  The counterpart of \eqref{Birkhoffb} in this context reads
\begin{equation}\label{Birkhoffa}
\left\{
\begin{aligned}
& \text{Find $p\in {\mathbb P}_N$ such that  for any  $u\in C^1(I),$}\\
& p(-1)=u(-1),\quad p'(x_j)=u'(x_j),\quad 1\le j\le N.
\end{aligned}\right.
\end{equation}
One verifies readily that $p(x)$ can be uniquely expressed by
\begin{equation}\label{birkhoffint}
p(x)=u(-1) B_0(x)+ \sum_{j=1}^N u'(x_j) B_j(x),\quad x\in [-1,1],
\end{equation}
if there exist $\{B_j\}_{j=0}^N\subseteq {\mathbb P}_N$  such that
\begin{align}
&B_0(-1)=1,\;\; B_0'(x_i)=0,\;\; 1\le i\le N;\;\; B_j(-1)=0, \;\;   B_j'(x_i)=\delta_{ij},\;\;  1\le i,j\le N. \label{H0basis}
\end{align}
Like  Theorem \ref{Birkhoffbasis2},  we can derive
\begin{equation}\label{interbasis}
B_0(x)=1;  \quad B_j(x)=\int_{-1}^x L_j (t)\, dt,\;\;\; 1\le  j\le N,
\end{equation}
where
\begin{equation}\label{ljxexp}
L_j(x)=\frac{Q_N(x)}{(x-x_j)Q_N'(x_j)},\quad Q_{N}(x)=c_N  \prod_{j=1}^N(x-x_j), \;\; c_N\not=0.
\end{equation}
%
%
%
      Let $\{l_j\}_{j=0}^N$ be the Lagrange basis polynomials associated with $\{x_j\}_{j=0}^N.$
Set  $b_{ij}:=B_{j}(x_i)$ and $d_{ij}:=l_j'(x_i).$  Define
\begin{equation}\label{Dmatrixas}
\bs B=(b_{ij})_{0\le i,j\le N},\;\; \bs B_{\rm in}=(b_{ij})_{1\le i,j\le N},\;\;  \bs D=(d_{ij})_{0\le i,j\le N},\;\;
\bs D_{\rm in}=(d_{ij})_{1\le i,j\le N}.
 \end{equation}

Like \eqref{BJrela200},   we have the following important properties.
\begin{theorem}\label{imprttrela} There hold
\begin{equation}\label{1storderinv}
\bs D_{\rm in} \bs B_{\rm in}={\bs I}_N,\quad  \widetilde  {\bs D}\bs B=\bs I_{N+1},
\end{equation}
where $\widetilde  {\bs D}$ is obtained by replacing the first row of
 $\bs D$ by $\bs e_1=(1,0,\cdots,0).$
\end{theorem}
 \begin{proof} For any $\phi\in {\mathbb P}_N,$ we write  $\phi(x)=\sum_{k=0}^N \phi(x_k) l_k(x),$ and
 \begin{equation}\label{BJrela0}
 \phi'(x)=\sum_{k=0}^N \phi(x_k) l_k'(x).
\end{equation}
Taking $\phi=B_j$ and setting $x=x_i,$ leads to
 \begin{equation}\label{BJrelaaa}
 B_j'(x_i)= \sum_{k=0}^N B_j(x_k) l_k'(x_i)=\sum_{k=0}^N d_{ik} b_{kj}.
\end{equation}
Thus, for $1\le i,j\le N,$ we obtain from $ B_j'(x_i)=\delta_{ij}$ and $b_{0j}=0$ that
 \begin{equation}\label{BJrela}
\delta_{ij}= \sum_{k=1}^N d_{ik} b_{kj},\quad 1\le i,j\le N,
\end{equation}
which implies $\bs D_{\rm in} \bs B_{\rm in}={\bs I}_N.$

Notice that the first column of $\bs B$ is $\bs e_1$ (cf.  \eqref{H0basis}), so
we verify from \eqref{BJrelaaa}-\eqref{BJrela} that  $ \widetilde {\bs  D}\bs B={\bs I}_{N+1}.$
 \end{proof}

As with Propositions \ref{prop:LGL}-\ref{ChebyGL},  we provide formulas to compute
 $\{B_j\}$ for Chebyshev- and Legendre-Gauss-Radau interpolation. To avoid repetition, we just give the derivation for the CGR case.
  \begin{prop}[{\bf Birkhoff interpolation at CGR points}]\label{ChebyGR} 
  The Birkhoff interpolation basis polynomials
  $\big\{B_j\big\}_{j=0}^N$ in \eqref{H0basis} at CGR points $\big\{x_j=-\cos(jh)\big\}_{j=0}^N,$
  $h=\frac{2\pi}{2N+1},$  are computed by
  \begin{equation}\label{Bjbass}
B_0(x)=1; \quad  B_j(x)=\sum_{k=0}^{N-1}\alpha_{kj} \partial_x^{-1} T_k(x), \;\;\;  1\le j\le N,
\end{equation}
where $\partial_x^{-1} T_k(x)$ is defined  in \eqref{chebyintegral}, and
\begin{equation}\label{alphakjformula}
\alpha_{kj}=  \frac{4}{ c_k(2N+1)}\big(T_k(x_j)-(-1)^{N+k}T_N(x_j)\big),
\end{equation}
with $c_0=2$ and $c_k=1$ for $k\ge 1.$
\end{prop}
\begin{proof} Writing $B_j'(x)=\sum_{k=0}^{N-1}\alpha_{kj}T_k(x),$ we derive from \eqref{Legorth} that
\begin{align*}
\alpha_{kj}&=\frac 2 {c_k\pi}\int_{-1}^1 \frac{B_j'(x)T_k(x)}{\sqrt{1-x^2}} dx=\frac 2 {c_k\pi}\Big(B_j'(-1)T_k(-1)\frac h 2+ T_k(x_j)h\Big),
\end{align*}
where we also used \eqref{H0basis} and the property that CGL quadrature is exact for all polynomials in ${\mathbb P}_{2N}$ (see e.g., \cite[Theorem 3.30]{ShenTangWang2011}).   Taking $k=N,$  we have from \eqref{Legorth} and \eqref{chebpm1} that   $\alpha_{kj}=0,$ and $B_j'(-1)=(-1)^{N+1} 2T_N(x_j).$ Thus \eqref{alphakjformula} follows.  Then direct integration leads to
$$
B_j(x)=\sum_{k=0}^{N-1}\alpha_{kj}\partial_x^{-1}T_k(x)+C.
$$
Since $\partial_x^{-1}T_k(-1)=0,$ we find $C=0$ from $B_j(-1)=0$ in \eqref{H0basis}.
\end{proof}

We can derive the formulas for computing $\{B_j\}$ at LGR points in a very similar fashion.
\begin{prop}[{\bf Birkhoff interpolation at LGR points}]\label{LegGR} Let $\{x_j,\omega_j\}_{j=0}^N$  be the LGR quadrature points
{\rm(}zeros of $P_N(x)+P_{N+1}(x)$ with  $x_0=-1${\rm)} and weights given by
 \begin{equation}\label{LGRweights}
 \omega_j=\frac 1 {(N+1)^2}\frac {1-x_j}{P_N^2(x_j)},\quad 0\le j\le N.
 \end{equation}
Then the Birkhoff interpolation basis polynomials
  $\big\{B_j\big\}_{j=0}^N$ in  \eqref{H0basis}  can be computed by
 \begin{equation}\label{BjbassLeg}
B_0(x)=1;\;\;\;  B_j(x)=\sum_{k=0}^{N-1} \alpha_{kj}\frac{\partial_x^{-1}P_{k}(x)}{\gamma_k},\; \;\; 1\le j\le N,
\end{equation}
where $\gamma_k=\frac 2 {2k+1},$   $\partial_x^{-1}P_{k}(x)$ is given in \eqref{Legenderivative}, and
\begin{equation}\label{entriescknew}
\alpha_{kj}=\big(P_k(x_j)-(-1)^{N+k}P_N(x_j)\big)\omega_j.
\end{equation}
 \end{prop}



%
With the new basis at our disposal, we now apply it to solve first-order IVPs.  Consider
\begin{equation}\label{firstordereqnss}
u'(x)+\gamma(x) u(x)=f(x), \;\;  x\in I;\quad u(-1)=u_-,
\end{equation}
where $\gamma(x)$ and $f(x)$ are given continuous functions on $I,$  and $u_-$ is a given constant.
  The collocation scheme at Gauss-Radau points for \eqref{firstordereqnss} is to  find $u_N\in {\mathbb P}_N$ such that
\begin{equation}\label{firstordereqn}
u_N'(x_j)+\gamma(x_j) u_N(x_j)=f(x_j), \quad 1\le j\le N;\quad u_N(-1)=u_-.
\end{equation}
The matrix form of \eqref{firstordereqn} under the Lagrange interpolation basis $\{l_j\}_{j=0}^N,$  reads
\begin{equation}\label{firstordereqnA}
\big(\bs D_{\rm in}+\bs \Lambda_N \big)\bs u=\bs f- u_- \bs d_0,
\end{equation}
where $\bs D_{\rm in}$ is defined in \eqref{Dmatrixas}, and
 \begin{equation}\label{defunfn}
 \begin{split}
&\bs u=\big(u_N(x_1), \cdots, u_N(x_N)\big)^t,\quad \bs f=\big(f(x_1), \cdots, f(x_N)\big)^t,\\
&\bs d_0=\big(l_0'(x_1),\cdots, l_0'(x_N)\big)^t,\quad \bs \Lambda_N ={\rm diag}\big(\gamma(x_1),\cdots,\gamma(x_N)\big).
\end{split}
 \end{equation}
Note that the condition number of the coefficient matrix in \eqref{firstordereqnA}   grows like $N^2$.

Under the new basis  $\{B_j\}_{j=0}^N,$ we find from \eqref{H0basis} the matrix form:
\begin{equation}\label{firstordereqnB}
\big(\bs I_{N} + \bs \Lambda_N \bs B_{\rm in}\big)\bs v=\bs f- u_- \bar {\bs \gamma},
\end{equation}
where $\bs B_{\rm in}$ is defined in \eqref{Dmatrixas},  $\bs f$ is the same as in \eqref{defunfn}, and
 \begin{equation}\label{defunfnvd}
 \bs v=\big(u_N'(x_1),\cdots, u_N'(x_N)\big)^t,\quad  \bar {\bs \gamma}=\big(\gamma(x_1),\cdots,\gamma(x_N)\big)^t.
 \end{equation}

As a  comparison,  we tabulate in Table \ref{GRtab} the condition numbers of
\eqref{firstordereqnA} (LCOL) and \eqref{firstordereqnB} (BCOL)   with $\gamma=1, x^3$ and various $N.$ As what we have observed from previous section,  the condition numbers of BCOL are independent of $N,$ while those of  LCOL grow like $N^2.$
\begin{table}[!th]
\caption{ Comparison of the condition numbers}\label{GRtab}
\vspace*{-6pt}
{\small
\begin{tabular}{|c|c|c|c|c|c|
c|c|c|}
    \hline
    \multirow{3}{*}{$N$} & \multicolumn{4}{|c|}{$\gamma = 1$} & \multicolumn{4}{|c|}{$\gamma = x^3$}\\
    \cline{2-9}
    & \multicolumn{2}{|c|}{Chebyshev} & \multicolumn{2}{|c|}{Legendre} & \multicolumn{2}{|c|}{Chebyshev} & \multicolumn{2}{|c|}{Legendre}\\
    \cline{2-9}
    & BCOL & LCOL & BCOL & LCOL & BCOL & LCOL & BCOL & LCOL\\
    \hline
    32 & 2.35 & 3.61e+02 & 2.35 & 4.67e+02 & 2.16 & 6.77e+02 & 2.14 & 8.86e+02\\
    \hline
    64 & 2.35 & 1.42e+03 & 2.35 & 1.98e+03 & 2.15 & 2.66e+03 & 2.15 & 3.74e+03\\
    \hline
    128 & 2.35 & 5.65e+03 & 2.34 & 8.45e+03 & 2.15 & 1.06e+04 & 2.14 & 1.59e+04\\
    \hline
    256 & 2.35 & 2.25e+04 & 2.35 & 3.59e+04 & 2.15 & 4.21e+04 & 2.15& 6.74e+04\\
    \hline
    512 & 2.35 & 8.98e+04 & 2.35 & 1.52e+05 & 2.15 & 1.68e+05 & 2.15 & 2.85e+05\\
    \hline
    1024 & 2.35 & 3.59e+05 & 2.35 & 6.40e+05 & 2.15 & 6.72e+05 & 2.15 & 1.20e+06\\
    \hline
\end{tabular}
}
\end{table}

We next consider \eqref{firstordereqnss} with $\gamma(x)=x^3, f(x)=20\sin (500 x^2)$ and
a highly oscillatory solution (see \cite[Section 2.5]{OlverTownsend}):
 \begin{equation}\label{prob1os}
    u(x) = 20\exp\Big(\frac{-x^4}{4}\Big)\int_{-1}^x \exp\Big(\frac{t^4}{4}\Big)\sin(500t^2)\, {d}t.
\end{equation}
In Figure \ref{accgraphs11} (left), we plot the exact solution \eqref{prob1os}  at $2000$ evenly-spaced points against the
numerical solution obtained by  BCOL with $N=640$.  In Figure \ref{accgraphs11} (right), we plot
the maximum pointwise errors of LCOL and BCOL for the Chebyshev case. It indicates that even for large $N,$  the BCOL is quite stable.
 \begin{figure}[!ht]
        \begin{center}
            \begin{minipage}{0.4\textwidth}
                \includegraphics[width=1\textwidth]{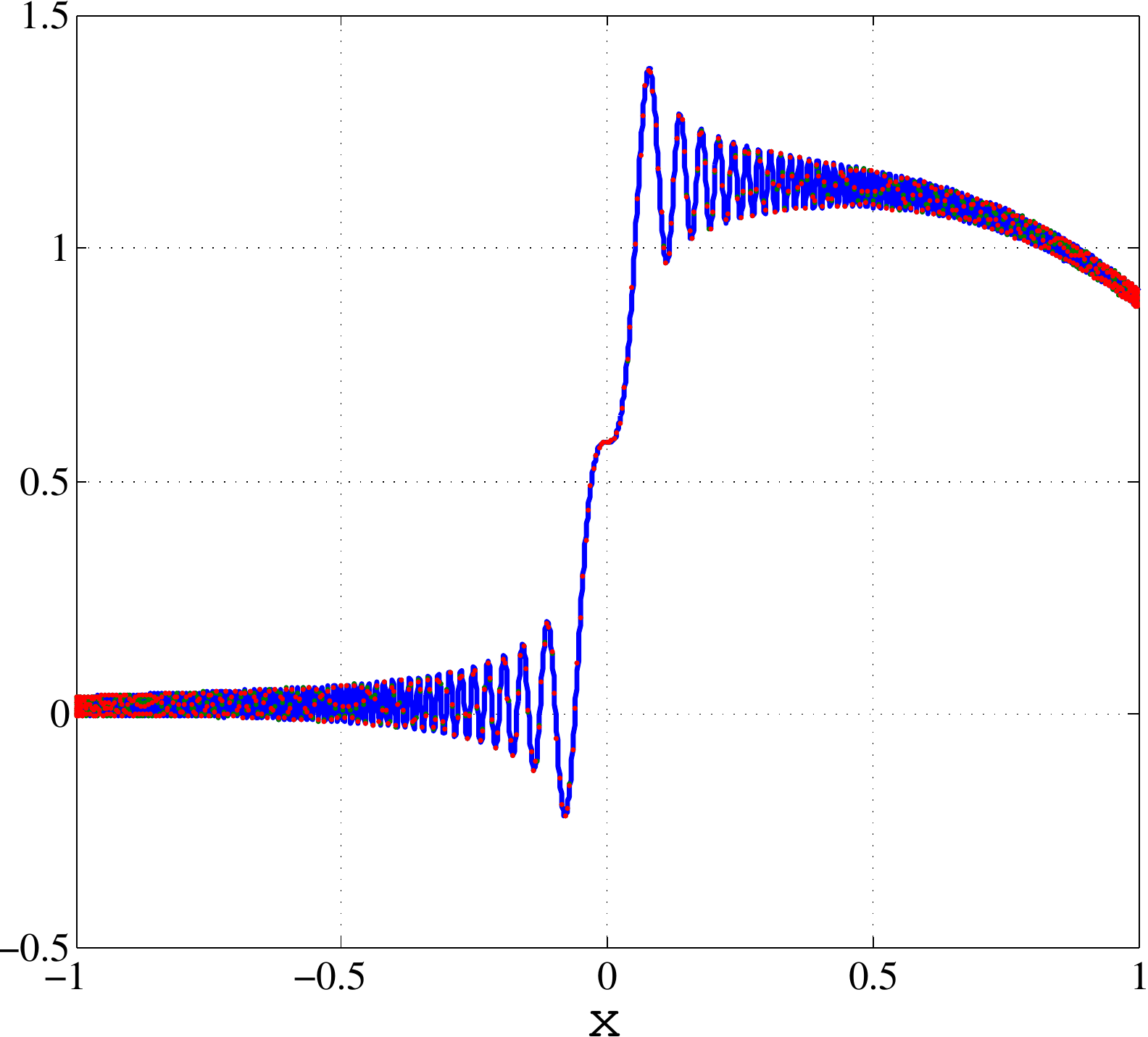}
            \end{minipage}\qquad
            \begin{minipage}{0.4\textwidth}
                \includegraphics[width=1\textwidth]{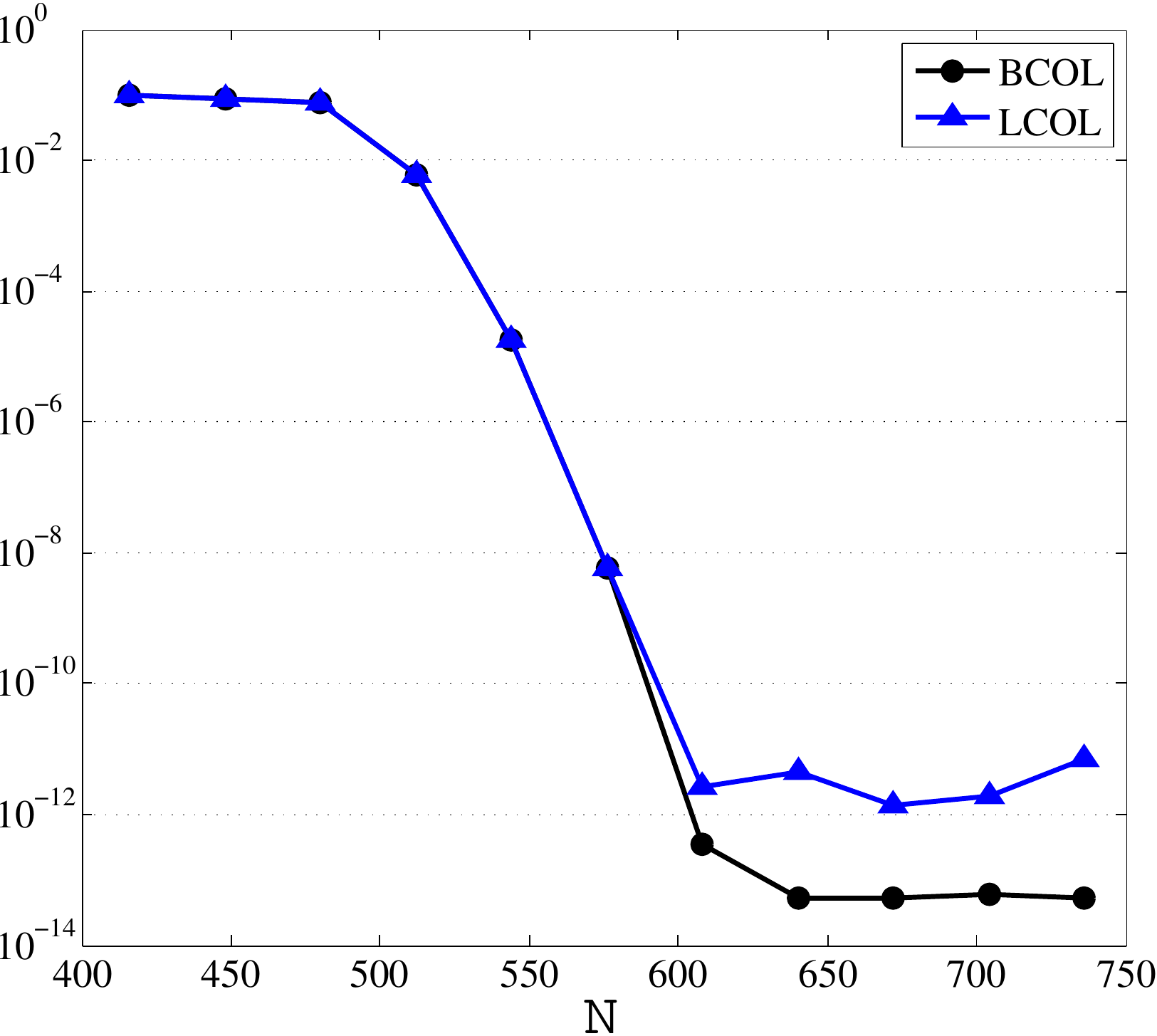}
            \end{minipage}
        \end{center}
        \caption{\small{Left: exact solution versus numerical solution.  Right: comparison of  numerical errors (Chebyshev).
        }}
        \label{accgraphs11}
    \end{figure}

\subsection{Higher order equations}  The proposed methods  can be directly extended to higher order BVPs.

\subsubsection{Third-order equations} For example, we consider
 $$
 -u'''(x)+r(x) u''(x)+s(x)u'(x)+t(x) u(x)=f(x),\quad x\in I; \quad u(\pm 1)=u_\pm,\;\; u'(1)=u_1.
 $$
As before, we associate it with a Birkhoff interpolation: Find $p\in {\mathbb P}_{N+1},$ such that
\begin{equation}\label{pxx3rd}
p(\pm 1)=u(\pm 1),\quad p'(1)=u'(1),\quad p'''(x_j)=u'''(x_j),\;\;\; 1\le j\le  N-1,
\end{equation}
where $\{x_j\}_{j=1}^{N-1}$ are interior LGL points.   Then
\begin{equation}\label{pxx3rdpx}
p(x)=u(-1)B_0(x)+\sum_{j=1}^{N-1} u'''(x_j) B_j(x)+u(1)B_N(x)+u'(1)B_{N+1}(x),
\end{equation}
where the basis polynomials $\big\{B_j(x)\big\}_{j= 0}^{N + 1}$ are defined by
\begin{align*}
&B_0(-1)=1,\quad B_0(1)=0,\quad B_0'(1)=0,\quad  B_0'''(x_i)=0,\;\;\; 1\le i\le N-1;\\
& B_j(-1)=0,\quad B_j(1)=0,\quad B_j'(1)=0, \quad   B_j'''(x_i)=\delta_{ij},\;\;\;  1\le i,j\le N-1;\\ 
&B_N(-1)=0,\quad B_N(1)=1,\quad B_N'(1)=0,\quad  B_N'''(x_i)=0,\;\;\; 1\le i\le N-1;\\
&B_{N+1}(-1)=0,\quad B_{N+1}(1)=0,\quad B_{N+1}'(1)=1,\quad  B_{N+1}'''(x_i)=0,\;\;\; 1\le i\le N-1.
\end{align*}
We can compute the basis and the associated pseudospectral integration matrices on  CGL and LGL points, which we leave to the interested readers.
Here, we just tabulate in Table \ref{GLtab3} the condition numbers of the new approach on CGL points.  In all cases,
the condition numbers are independent of $N.$
{\small
\begin{table}[!th]
\caption{\small Condition numbers of (\ref{pxx3rd}) on CGL points}
\vspace*{-6pt}
\begin{tabular}{|c|c|c|c|c|}
    \hline
    $N$ & $r=s=0, t=1$ & $r=0, s=t=1$ & $s=0, r=t=1$ & $r=s=t=1$ \\
    \hline
128 & 1.16 & 1.56 & 2.22 & 1.80\\
    \hline
256 & 1.16 & 1.56 & 2.22 & 1.80\\
    \hline
512 & 1.16 & 1.56 & 2.23 & 1.80\\
    \hline
1024 & 1.16 & 1.56 & 2.23 & 1.80\\
    \hline
\end{tabular}\label{GLtab3}
\end{table}
}

We next apply the well-conditioned collocation method to solve the Korteweg-de Vires (KdV) equation:
\begin{equation}\label{KdV3}
	\partial_t u + u \partial_x u + \partial_x^3 u = 0; \quad u(x,0) = u_0(x),
\end{equation}
with the exact soliton solution
\begin{equation}\label{KdV3s}
	u(x, t) = 12 \kappa^2{\rm sech}^2(\kappa(x - 4\kappa^2t - x_0)),
\end{equation}
where $\kappa$ and $x_0$ are constants. Since  the solution decays exponentially, we can approximate the initial value problems by imposing
homogeneous boundary conditions over $x\in(-L, L)$ as long as the soliton wave does not reach the boundaries. Let $\tau$ be the time step size,
and $\{\xi_j=Lx_j\}_{j=0}^N$  with $\{x_j\}_{j=0}^N$ being  CGL points.   Then we adopt the Crank-Nicolson leap-frog scheme
in time and the new collocation method in space, that is, find   $u_N^{k+1}\in {\mathcal P}_{N+1}$ such that
for $1\le j\le N-1,$
\begin{equation}\label{KdV3d}
\begin{split}
&\frac{u^{k + 1}_{N}(\xi_j) - u^{k - 1}_N (\xi_j)}{2\tau}  +\partial_x^3 \bigg(\frac{ u_N^{k + 1} +u_N^{k - 1}}{2}\bigg)(\xi_j)
= - \partial_x u_N^k(\xi_j)u_N^k(\xi_j),\\
& u_N^k(\pm L)=\partial_x u_N^k (L)=0,\quad k\ge 0.
\end{split}
\end{equation}
Here, we take $\kappa = 0.3$, $x_0 = -20, L = 50$ and $\tau=0.001.$
We depict in Figure \ref{KdV3graphs} (left) the numerical evolution of the solution with $t\le 50$ and $N=160.$
In Figure \ref{KdV3graphs} (right), we plot the maximum point-wise errors for various $N$ at $t=1,50.$
We see the errors decay exponentially, and the scheme is stable. Indeed, the proposed collocation method
produces  very accurate and stable solution as the well-conditioned dual-Petrov-Galerkin method in
\cite{Shen02b}.

\begin{figure}[!ht]
        \begin{center}
            \begin{minipage}{0.42\textwidth}
                \includegraphics[width=1\textwidth]{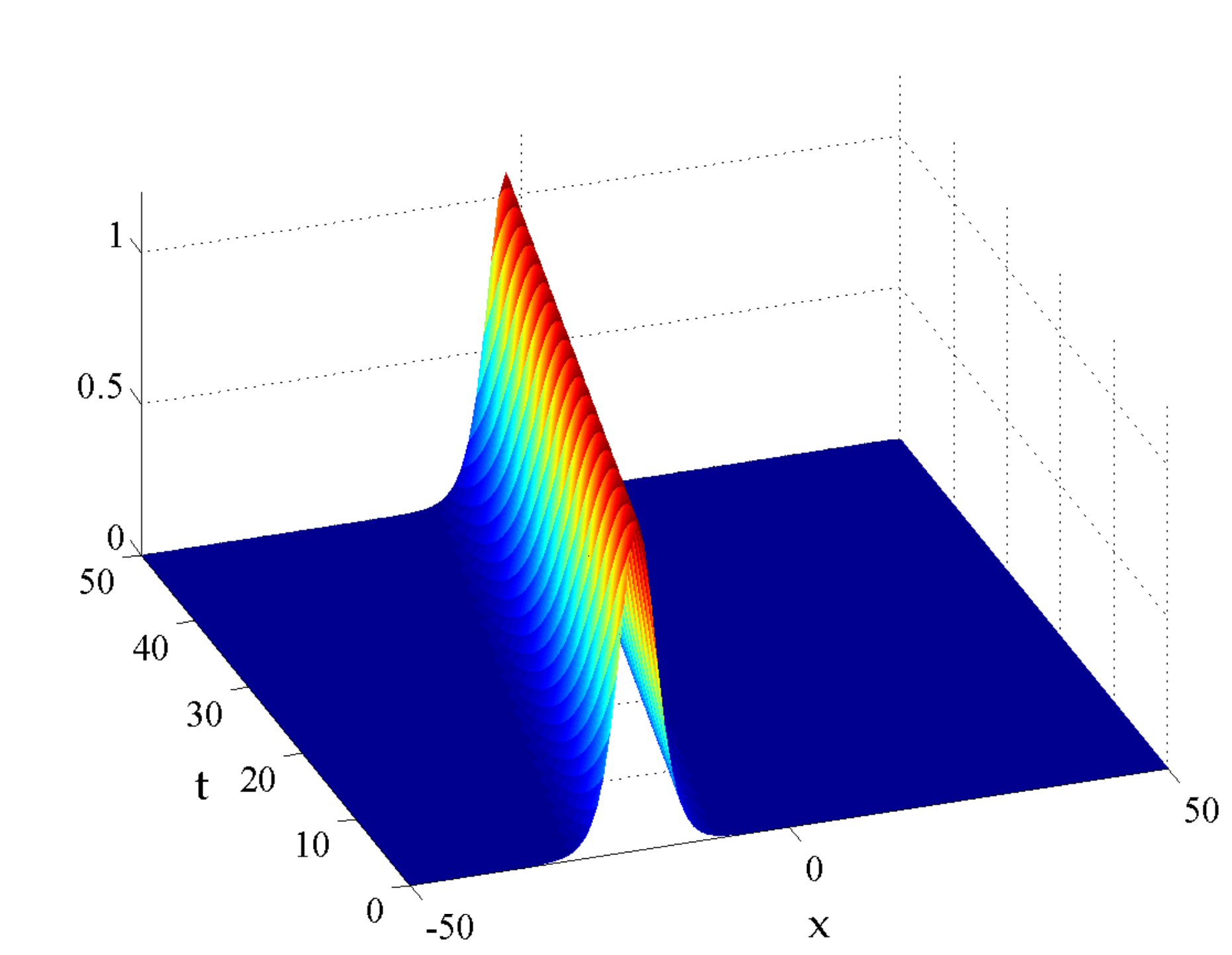}
            \end{minipage}\qquad
            \begin{minipage}{0.4\textwidth}
                \includegraphics[width=0.95\textwidth]{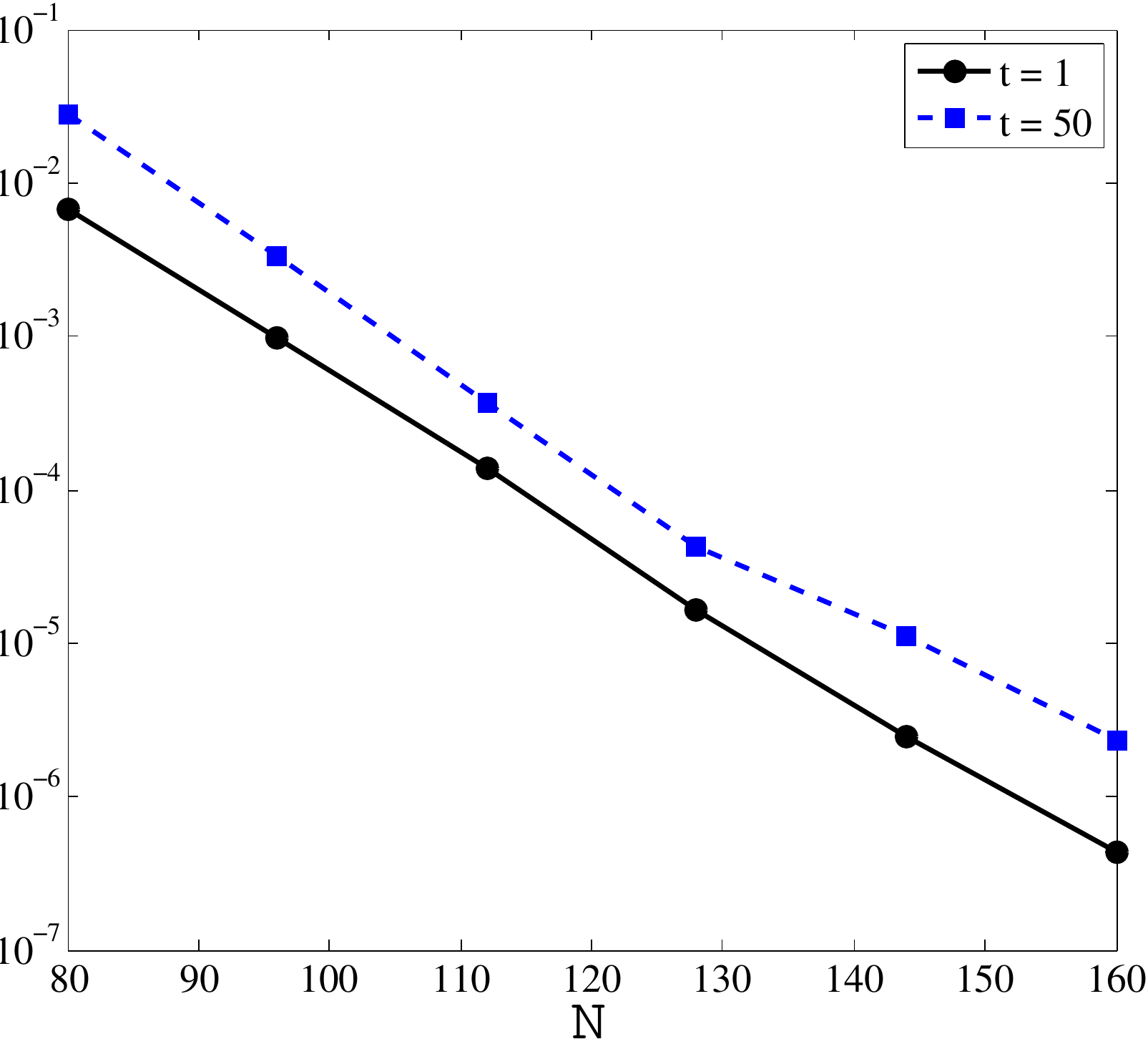}
            \end{minipage}
        \end{center}
        \caption{\small{Left: time evolution of numerical solution for $N = 160$.  Right: maximum absolute error at interior collocation points at given $t$ for given $N$.}}
        \label{KdV3graphs}
    \end{figure}

\subsubsection{Fifth-order equations}
We can  extend the  notion of Birkhoff interpolation  and derive the new basis for fifth-order problem straightforwardly.
Here, we omit the details, but just  test the method on the problem:
\begin{equation}\label{fifth}
	u^{(5)}(x) + \sin(10x)  u'(x) + x u(x) = f(x), \quad x \in I;  \quad u(\pm1) = u'(\pm1) = u''(1) = 0,
\end{equation}
with exact solution	$u(x) = \sin^3(\pi x).$   Here, we compare the usual Lagrange collocation method (LCOL), the new Birkhoff collocation
(BCOL) scheme at CGL points, and the special collocation  method (SCOL).  We refer to the SCOL as in \cite[Page 218]{ShenTangWang2011}, which
is based on the interpolation problem: Find $p\in {\mathbb P}_{N+3}$ such that
\begin{align*}
& p(y_j)=u(y_j),\;\;\; 1\le j\le N-1; \;\; p^{(k)}(\pm 1)=u^{(k)}(\pm 1),\;\;\; k=0,1;\;\;\;  p''(1)=u''(1),
\end{align*}
where $\{y_j\}_{j=1}^{N-1}$ are zeros of the Jacobi polynomial $P_{N-1}^{(3,2)}(x).$

We plot in Figure \ref{fifthgraph} (left) convergence behavior of three methods, which clearly indicates the new approach
is well-conditioned and significantly superior to the other two.
\begin{figure}[!ht]
        \begin{center}
            \begin{minipage}{0.4\textwidth}
                \includegraphics[width=1\textwidth]{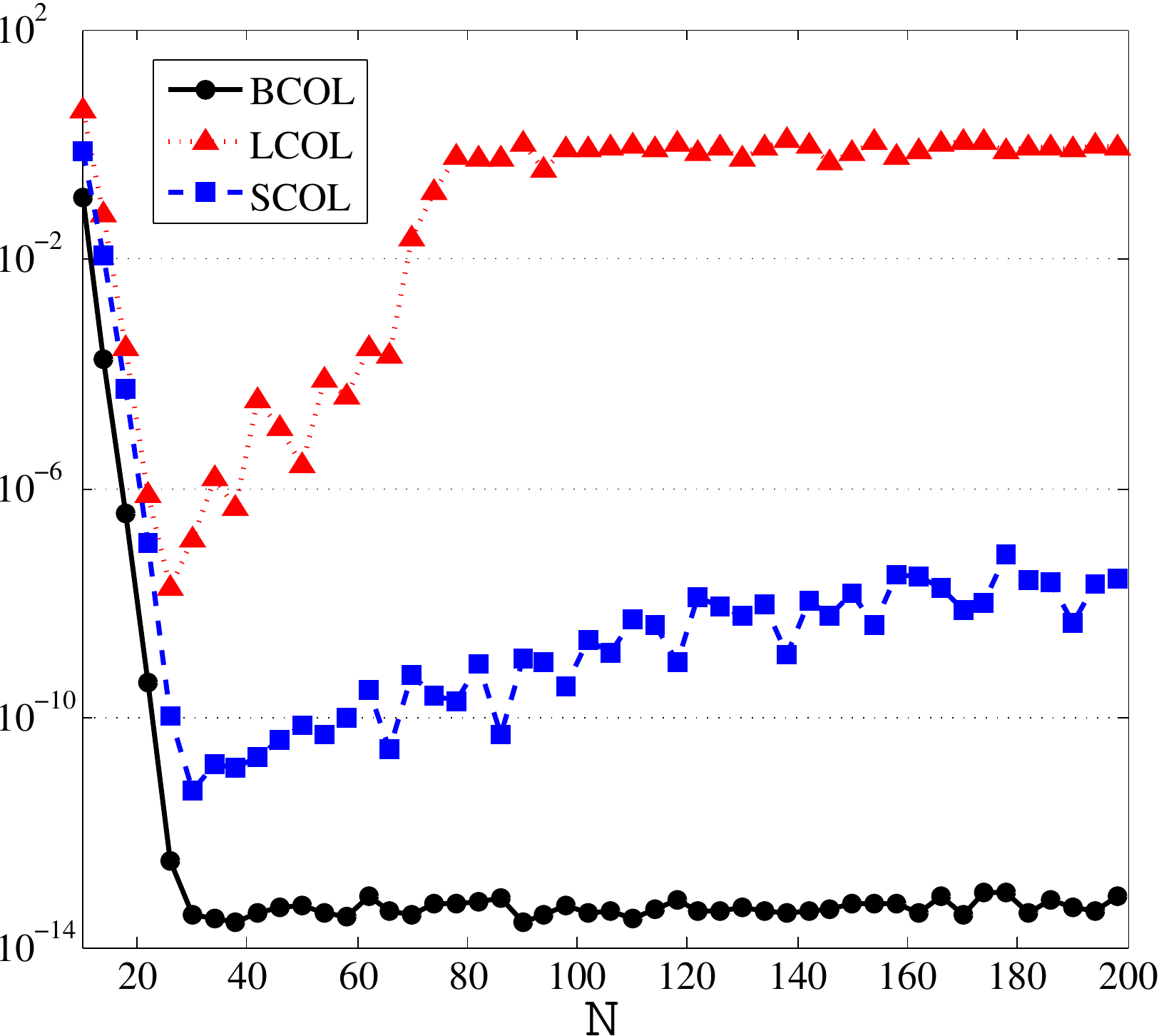}
            \end{minipage}\qquad
                        \begin{minipage}{0.4\textwidth}
                \includegraphics[width=1\textwidth]{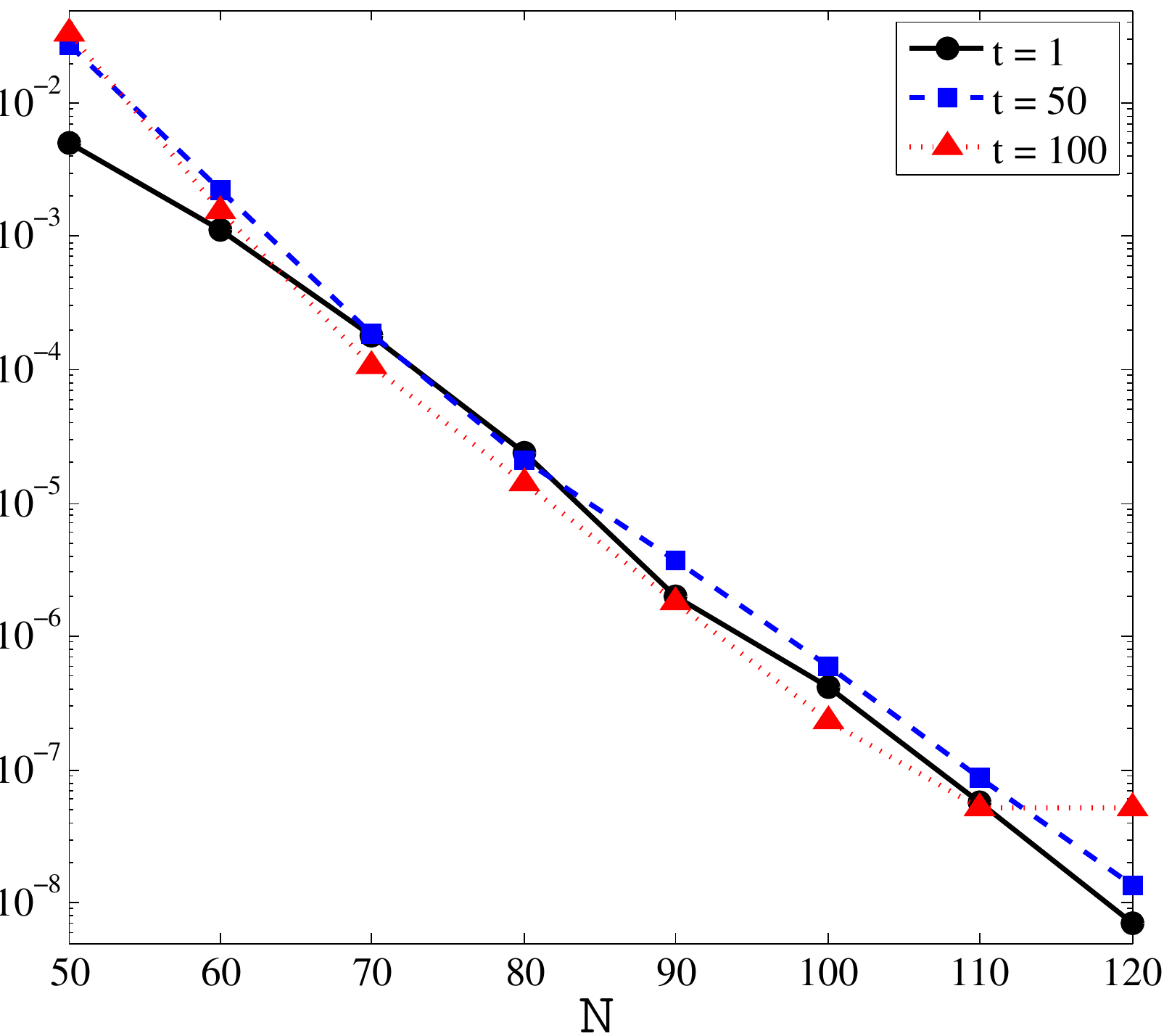}
            \end{minipage}
        \end{center}
        \caption{\small{Comparison of three collocation schemes (left), and  maximum pointwise errors
        of the Crank-Nicolson-leap-frog and BCOL for fifth-order KdV equation (right).}}
        \label{fifthgraph}
    \end{figure}
We  also apply the new method in space to solve the fifth-order KdV equation:
\begin{equation}\label{KdV5}
	\partial_t u + \gamma u \partial_x u + \nu \partial_x^3 u  - \mu \partial^5_x u = 0,\quad  u(x,0) = u_0(x).
\end{equation}
For  $\gamma \neq 0$,  and  $\mu\nu > 0,$ it has the exact soliton solution (cf. \cite[Page 233]{ShenTangWang2011} and the original references therein):
\begin{equation}\label{KdV5s}
	u(x, t) = \eta_0 + \frac{105 \nu^2}{169\mu\gamma} {\rm sech}^4\left(\sqrt{\frac{\nu}{52\mu}}\left[x - \left(\gamma\eta_0 + \frac{36\nu^2}{169\mu}\right)t - x_0\right]\right),
\end{equation}
where  $\eta_0$ and $x_0$ are any constants.  As with \eqref{KdV3d}, we use the Crank-Nicolson-leap-frog in time and new collocation
method in space.   In Figure  \ref{fifthgraph} (right), we depict the maximum pointwise errors at CGL points for
 \eqref{KdV5}-\eqref{KdV5s} with $\mu = \gamma = 1$, $\nu = 1.1$, $\eta_0 = 0$, $x_0 = -10,$  $L = 50$ and
 $\tau = 0.001.$  It indicates that the scheme is stable and accurate, which is comparable to the well-conditioned
 dual-Petrov-Galerkin scheme (cf. \cite[Chapter 6]{ShenTangWang2011}).

\subsection{Multi-dimensional cases} 
For example,  we consider the two-dimensional BVP:
\begin{equation}\label{2dbvp}
    \Delta u-  \gamma u = f \;\;\;  \text{ in  }\;\;  \Omega=(-1,1)^2; \;\;\;  u = 0\;\; \text{ on }\;\; \partial\Omega,
\end{equation}
where $\gamma\ge 0$ and $f\in C(\Omega).$  The collocation scheme is on  tensorial LGL points:  find $u_N(x,y)\in {\mathbb P}_N^2$ such that
\begin{equation}\label{2dbvpcol}
    \big(\Delta u_N - \gamma u_N\big)(x_i,y_j) = f(x_i,y_j),\;\;  1\le i,j\le N-1; \;\;\;  u_N = 0\;\; \text{ on }\;\; \partial\Omega,
\end{equation}
where $\{x_i\}$ and $\{y_j\}$ are LGL points.
As with the spectral-Galerkin method \cite{Shen94b,She.W07b}, we use the matrix decomposition (or diagonalization)  technique (see  \cite{LRT64}).  We illustrate the idea by using  partial  diagonalization  (see \cite[Section 8.1]{ShenTangWang2011}). Write
$$u_N(x,y)=\sum_{k,l=1}^{N-1} u_{kl} B_k(x)B_l(y),$$
and obtain from \eqref{2dbvpcol} the system:
\begin{equation}\label{odesystm}
\bs U\bs B_{\rm in}^{t}+\bs B_{\rm in}\bs U-\gamma \bs B_{\rm in}\bs U\bs B_{\rm in}^t=\bs F,
\end{equation}
where $\bs U=(u_{kl})_{1\le k,l\le N-1}$ and $\bs F=(f_{kl})_{1\le k,l\le N-1}.$
We consider the generalized eigen-problem:
$$ \bs B_{\rm in}\,\bs x =\lambda \big(\bs I_{N-1} -\gamma \bs B_{\rm in}\big)\bs x. $$
We know from  Proposition \ref{addprof} and Remark \ref{Chebycond} that the eigenvalues are distinct. Let $\bs \Lambda$ be the diagonal
matrix of the eigenvalues, and  $\bs E$ be the matrix whose columns are  the corresponding eigenvectors.
Then we have
$$
\bs B_{\rm in}\bs E= \big(\bs I_{N-1} -\gamma \bs B_{\rm in}\big)\bs E \bs \Lambda.
$$
We describe  the partial diagonalization (see \cite[Section 8.1]{ShenTangWang2011}). Set  $\bs U=\bs E\bs V.$  Then \eqref{odesystm}
becomes
\begin{equation}\label{eqnssab0}
\bs V \bs B_{\rm in}^t +\bs \Lambda \bs V=\bs G:=\bs E^{-1} \big(\bs I_{N-1} -\gamma \bs B_{\rm in}\big)^{-1} \bs F.
\end{equation}
Taking transpose of the above equation leads to
\begin{equation}\label{eqnssab}
\bs B_{\rm in} \bs V^t +\bs V^t \bs \Lambda =\bs G^t.
\end{equation}
Let $\bs v_p$ be the transpose of  $p$-th row of $\bs V,$ and likewise for $\bs g_p.$ Then we solve the systems:
\begin{equation}\label{eqnssabc}
\big(\bs B_{\rm in} +\lambda_p \bs I_{N-1}\big)\bs v_p =\bs g_p, \quad p=1,2,\cdots,N-1.
\end{equation}
As shown in Section \ref{sect2}, the coefficient matrix is well-conditioned.  Note that this process can be extended to three dimensions
straightforwardly.

%
%

As a numerical illustration, we consider \eqref{2dbvp} with $\gamma = 0$ and $u(x, y) = \sin(4\pi x)\sin(4\pi y).$
In Figure \ref{accuracygraph2D},  we graph the maximum pointwise  errors against various $N$ of the new approach,
which is comparable to the spectral-Galerkin approach in \cite{Shen94b}.
 \begin{figure}[!ht]
        \begin{center}
            \begin{minipage}{0.4\textwidth}
                \includegraphics[width=1\textwidth]{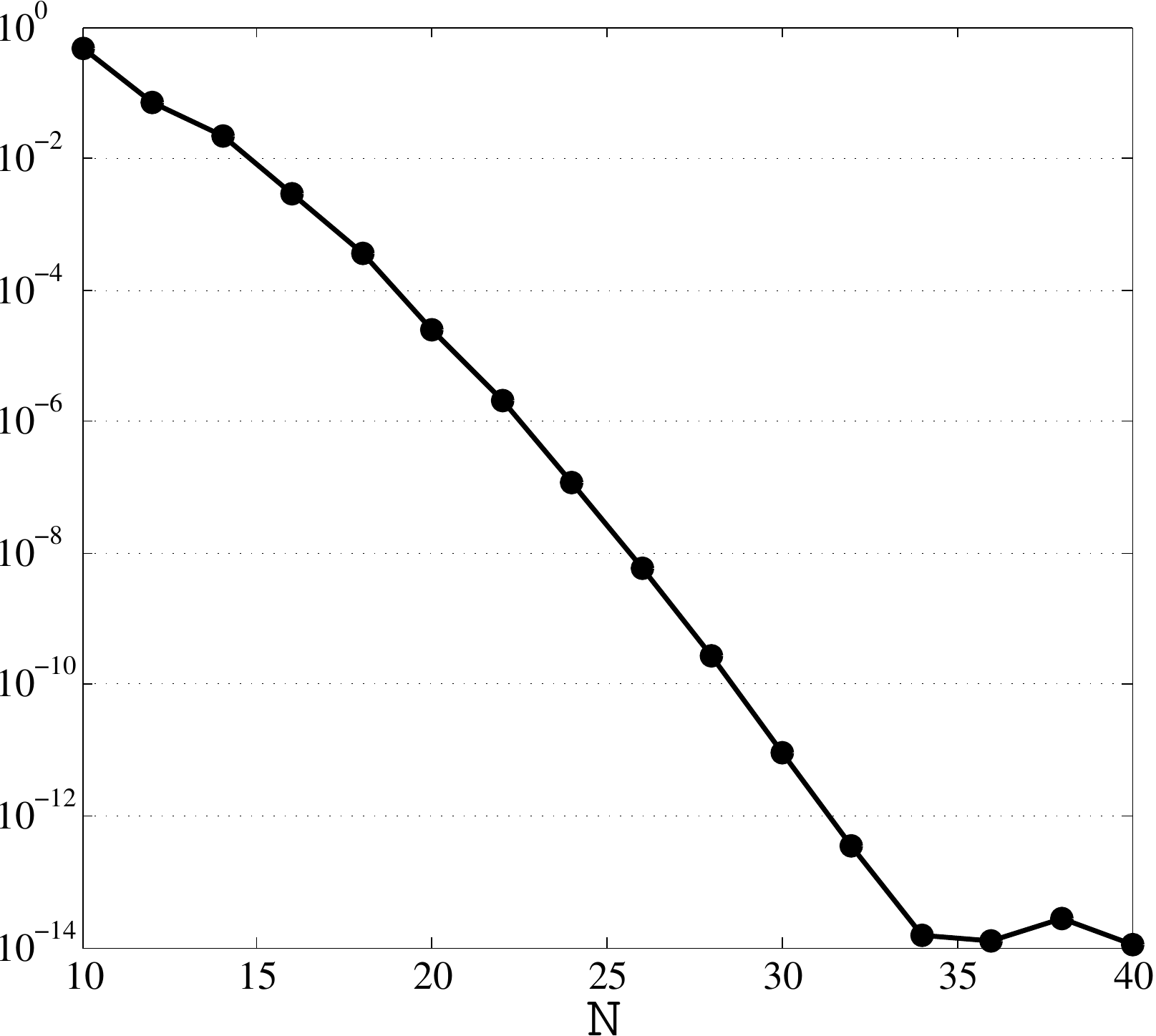}
            \end{minipage}\qquad
            \begin{minipage}{0.4\textwidth}
                \includegraphics[width=1\textwidth]{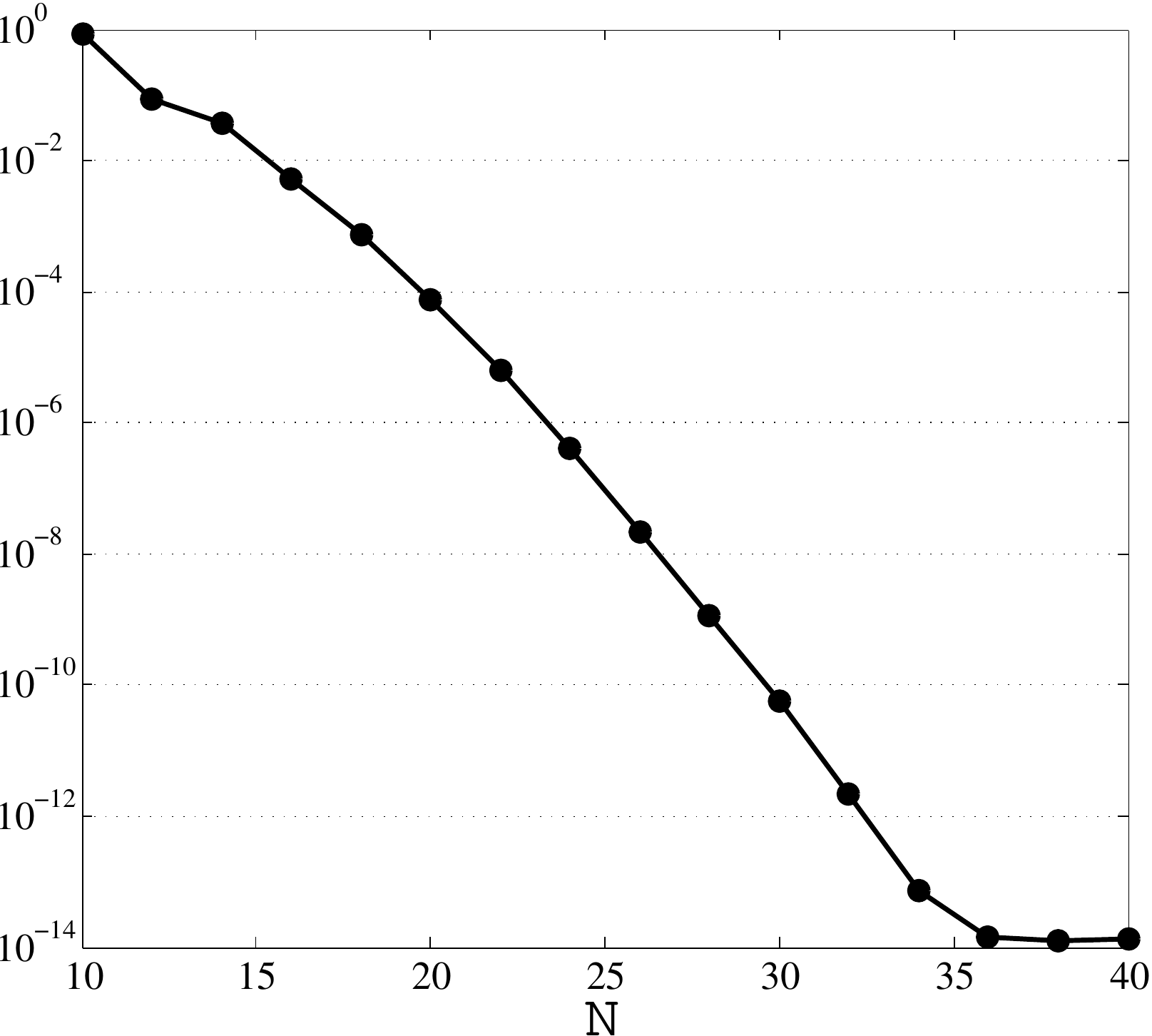}
            \end{minipage}
        \end{center}
        \caption{\small{Maximum pointwise  errors. Left: LGL; right: CGL.}}
        \label{accuracygraph2D}
    \end{figure}

%

\newpage
\noindent\underline{\large\bf Concluding remarks}
\vskip 6pt

In this paper, we tackled the longstanding issue with ill-conditioning of collocation/pseudospectral methods from
a new perspective. More precisely, we considered special Birkhoff interpolation problems that produced dual nature basis functions.
Firstly, the collocation systems under the new basis are well-conditioned, and the matrix corresponding to the
highest derivative of the equation  is diagonal or identity. The new collocation approach could be viewed as the analogue of the well-conditioned
Galerkin method in \cite{Shen94b}. Secondly, this approach led to optimal integration preconditioners for usual collocation schemes
based on Lagrange interpolation.  For the first time, we introduced in this paper the notion of pseudospectral integration matrix.

\vskip 10pt
\noindent\underline{\large\bf Acknowledgement}
\vskip 6pt

The first author would like to thank Prof. Benyu Guo and Prof. Jie Shen  for fruitful discussions, and
thank Prof. Zhimin Zhang for the stimulating Birkhoff interpolation problem considered in the recent paper  \cite{ZhangInp2012}.

\bibliography{Refcol}

\begin{thebibliography}{10}

\bibitem{Boulmezaoud2007}
T.Z. Boulmezaoud and J.M. Urquiza.
\newblock On the eigenvalues of the spectral second order differentiation
  operator and application to the boundary observability of the wave equation.
\newblock {\em J. Sci. Comput.}, 31(3):307--345, 2007.

\bibitem{Boyd01}
J.P. Boyd.
\newblock {\em {C}hebyshev and {F}ourier {S}pectral {M}ethods}.
\newblock Dover Publications Inc., 2001.

\bibitem{Can09}
C.~Canuto.
\newblock High-order methods for {PDE}s: recent advances and new perspectives.
\newblock In {\em I{CIAM} 07---6th {I}nternational {C}ongress on {I}ndustrial
  and {A}pplied {M}athematics}, pages 57--87. Eur. Math. Soc., Z\"urich, 2009.

\bibitem{CA10}
C.~Canuto, P.~Gervasio, and A.~Quarteroni.
\newblock Finite-element preconditioning of {G}-{NI} spectral methods.
\newblock {\em SIAM J. Sci. Comput.}, 31(6):4422--4451, 2009/10.

\bibitem{CHQZ06}
C.~Canuto, M.Y. Hussaini, A.~Quarteroni, and T.A. Zang.
\newblock {\em {Spectral Methods: Fundamentals in Single Domains}}.
\newblock Springer, Berlin, 2006.

\bibitem{CA85}
C.~Canuto and A.~Quarteroni.
\newblock Preconditioned minimal residual methods for {C}hebyshev spectral
  calculations.
\newblock {\em J. Comput. Phys.}, 60(2):315--337, 1985.

\bibitem{ChenS12}
F.~Chen and J.~Shen.
\newblock Efficient spectral-{G}alerkin methods for systems of coupled
  second-order equations and their applications.
\newblock {\em J. Comput. Phys.}, 231(15):5016--5028, 2012.

\bibitem{clenshaw1957numerical}
C.W. Clenshaw.
\newblock The numerical solution of linear differential equations in
  {C}hebyshev series.
\newblock In {\em Mathematical Proceedings of the Cambridge Philosophical
  Society}, volume~53, pages 134--149. Cambridge Univ Press, 1957.

\bibitem{CoL10}
F.A. Costabile and E.~Longo.
\newblock A {B}irkhoff interpolation problem and application.
\newblock {\em Calcolo}, 47(1):49--63, 2010.

\bibitem{coutsias1996integration}
E.~Coutsias, T.~Hagstrom, J.S. Hesthaven, and D.~Torres.
\newblock Integration preconditioners for differential operators in spectral
  $\tau$-methods.
\newblock In {\em Proceedings of the Third International Conference on Spectral
  and High Order Methods, Houston, TX}, pages 21--38, 1996.

\bibitem{CHT96}
E.A. Coutsias, T.~Hagstrom, and D.~Torres.
\newblock An efficient spectral method for ordinary differential equations with
  rational function coefficients.
\newblock {\em Math. Comp.}, 65(214):611--635, 1996.

\bibitem{Dev.M85}
M.O. Deville and E.H. Mund.
\newblock {C}hebyshev pseudospectral solution of second-order elliptic
  equations with finite element preconditioning.
\newblock {\em J. Comput. Phys.}, 60:517--533, 1985.

\bibitem{Dev.M90}
M.O. Deville and E.H. Mund.
\newblock Finite element preconditioning for pseudospectral solutions of
  elliptic problems.
\newblock {\em SIAM J. Sci. Stat. Comput.}, 11:311--342, 1990.

\bibitem{Dris10}
T.A. Driscoll.
\newblock Automatic spectral collocation for integral, integro-differential,
  and integrally reformulated differential equations.
\newblock {\em J. Comput. Phys.}, 229(17):5980--5998, 2010.

\bibitem{DBT08}
T.A. Driscoll, F.~Bornemann, and L.N. Trefethen.
\newblock The {C}hebop system for automatic solution of differential equations.
\newblock {\em BIT}, 48(4):701--723, 2008.

\bibitem{El69}
S.E. El-Gendi.
\newblock Chebyshev solution of differential, integral and integro-differential
  equations.
\newblock {\em Comput. J.}, 12:282--287, 1969/1970.

\bibitem{Elbarbary06}
M.E. Elbarbary.
\newblock Integration preconditioning matrix for ultraspherical pseudospectral
  operators.
\newblock {\em SIAM J. Sci. Comput.}, 28(3):1186--1201 (electronic), 2006.

\bibitem{ElS13}
K.T. Elgindy and K.A. Smith-Miles.
\newblock Solving boundary value problems, integral, and integro-differential
  equations using {G}egenbauer integration matrices.
\newblock {\em J. Comput. Appl. Math.}, 237(1):307--325, 2013.

\bibitem{EZGu99}
A.~Ezzirani and A.~Guessab.
\newblock A fast algorithm for {G}aussian type quadrature formulae with mixed
  boundary conditions and some lumped mass spectral approximations.
\newblock {\em Math. Comp.}, 68(225):217--248, 1999.

\bibitem{Forn96}
B.~Fornberg.
\newblock {\em A Practical Guide to Pseudospectral Methods}.
\newblock Cambridge University Press, 1996.

\bibitem{FG88}
D.~Funaro and D.~Gottlieb.
\newblock A new method of imposing boundary conditions in pseudospectral
  approximations of hyperbolic equations.
\newblock {\em Math. Comp.}, 51(184):599--613, 1988.

\bibitem{ghoreishi2004tau}
F.~Ghoreishi and S.M. Hosseini.
\newblock The {T}au method and a new preconditioner.
\newblock {\em J. Comput. Appl. Math.}, 163(2):351--379, 2004.

\bibitem{gottlieb1977numerical}
D.~Gottlieb and S.A. Orszag.
\newblock {\em {Numerical Analysis of Spectral Methods: Theory and
  Applications}}.
\newblock Society for Industrial Mathematics, 1977.

\bibitem{Greengard91}
L.~Greengard.
\newblock Spectral integration and two-point boundary value problems.
\newblock {\em SIAM J. Numer. Anal.}, 28(4):1071--1080, 1991.

\bibitem{Guobk98}
B.Y. Guo.
\newblock {\em Spectral Methods and Their Applications}.
\newblock World Scientific Publishing Co. Inc., River Edge, NJ, 1998.

\bibitem{Guo.SW06}
B.Y. Guo, J.~Shen, and L.L. Wang.
\newblock Optimal spectral-{G}alerkin methods using generalized {J}acobi
  polynomials.
\newblock {\em J. Sci. Comput.}, 27(1-3):305--322, 2006.

\bibitem{Hesthaven98}
J.~Hesthaven.
\newblock Integration preconditioning of pseudospectral operators. {I}. {B}asic
  linear operators.
\newblock {\em SIAM J. Numer. Anal.}, 35(4):1571--1593, 1998.

\bibitem{HGG07}
J.~Hesthaven, S.~Gottlieb, and D.~Gottlieb.
\newblock {\em Spectral Methods for Time-Dependent Problems}.
\newblock Cambridge Monographs on Applied and Computational Mathematics.
  Cambridge, 2007.

\bibitem{Kim.P96}
S.D. Kim and S.V. Parter.
\newblock Preconditioning {C}hebyshev spectral collocation method for elliptic
  partial differential equations.
\newblock {\em SIAM J. Numer. Anal.}, 33(6):2375--2400, 1996.

\bibitem{Kim.P97}
S.D. Kim and S.V. Parter.
\newblock Preconditioning {C}hebyshev spectral collocation by finite difference
  operators.
\newblock {\em SIAM J. Numer. Anal.}, 34(3):939--958, 1997.

\bibitem{LivP10}
P.W. Livermore.
\newblock Galerkin orthogonal polynomials.
\newblock {\em J. Comput. Phys.}, 229(6):2046--2060, 2010.

\bibitem{BirkhoffBk}
G.G. Lorentz, K.~Jetter, and S.D. Riemenschneider.
\newblock {\em Birkhoff Interpolation}, volume~19 of {\em Encyclopedia of
  Mathematics and its Applications}.
\newblock Addison-Wesley Publishing Co., Reading, Mass., 1983.

\bibitem{LRT64}
R.E. Lynch, J.R. Rice, and D.H. Thomas.
\newblock Direct solution of partial differential equations by tensor product
  methods.
\newblock {\em Numer. Math.}, 6:185--199, 1964.

\bibitem{MM02}
B.~Mihaila and I.~Mihaila.
\newblock Numerical approximations using {C}hebyshev polynomial expansions:
  {E}l-{G}endi's method revisited.
\newblock {\em J. Phys. A}, 35(3):731--746, 2002.

\bibitem{Muite10}
B.K. Muite.
\newblock A numerical comparison of {C}hebyshev methods for solving fourth
  order semilinear initial boundary value problems.
\newblock {\em J. Comput. Appl. Math.}, 234(2):317--342, 2010.

\bibitem{OlverTownsend}
S.~Olver and A.~Townsend.
\newblock A fast and well-conditioned spectral method.
\newblock {\em To appear in SIAM Review (also see arXiv:1202.1347v2)}, 2013.

\bibitem{Shen94b}
J.~Shen.
\newblock Efficient spectral-{G}alerkin method {I}. direct solvers for second-
  and fourth-order equations by using {L}egendre polynomials.
\newblock {\em SIAM J. Sci. Comput.}, 15(6):1489--1505, 1994.

\bibitem{Shen02b}
J.~Shen.
\newblock A new dual-{P}etrov-{G}alerkin method for third and higher odd-order
  differential equations: Application to the {KDV} equation.
\newblock {\em SIAM J. Numer. Anal}, 41(5):1595--1619, 2003.

\bibitem{ShenTangWang2011}
J.~Shen, T.~Tang, and L.L. Wang.
\newblock {\em {Spectral Methods: Algorithms, Analysis and Applications}},
  volume~41 of {\em Series in Computational Mathematics}.
\newblock Springer-Verlag, Berlin, Heidelberg, 2011.

\bibitem{She.W07b}
J.~Shen and L.L. Wang.
\newblock Fourierization of the {L}egendre-{G}alerkin method and a new
  space-time spectral method.
\newblock {\em Appl. Numer. Math.}, 57(5-7):710--720, 2007.

\bibitem{shi2003theory}
Y.G. Shi.
\newblock {\em {Theory of Birkhoff Interpolation}}.
\newblock Nova Science Pub Incorporated, 2003.

\bibitem{szeg75}
G.~Szeg\"o.
\newblock {\em Orthogonal Polynomials (Fourth Edition)}.
\newblock AMS Coll. Publ., 1975.

\bibitem{Tref00}
L.N. Trefethen.
\newblock {\em {Spectral Methods in {MATLAB}}}, volume~10 of {\em Software,
  Environments, and Tools}.
\newblock Society for Industrial and Applied Mathematics (SIAM), Philadelphia,
  PA, 2000.

\bibitem{TrTr87}
L.N. Trefethen and M.R. Trummer.
\newblock An instability phenomenon in spectral methods.
\newblock {\em SIAM J. Numer. Anal.}, 24(5):1008--1023, 1987.

\bibitem{WangG09}
L.L. Wang and B.Y. Guo.
\newblock Interpolation approximations based on
  {G}auss-{L}obatto-{L}egendre-{B}irkhoff quadrature.
\newblock {\em J. Approx. Theory}, 161(1):142--173, 2009.

\bibitem{weideman2000matlab}
J.A. Weideman and S.C. Reddy.
\newblock A {MATLAB} differentiation matrix suite.
\newblock {\em ACM Transactions on Mathematical Software (TOMS)},
  26(4):465--519, 2000.

\bibitem{Weideman1988}
J.A.C. Weideman and L.N. Trefethen.
\newblock The eigenvalues of second-order spectral differentiation matrices.
\newblock {\em SIAM J. Numer. Anal.}, 25(6):1279--1298, 1988.

\bibitem{Welfert1994}
B.D. Welfert.
\newblock On the eigenvalues of second-order pseudospectral differentiation
  operators.
\newblock {\em Comput. Methods Appl. Mech. Engrg.}, 116(1-4):281--292, 1994.
\newblock ICOSAHOM'92 (Montpellier, 1992).

\bibitem{Zebib84}
A.~Zebib.
\newblock A {C}hebyshev method for the solution of boundary value problems.
\newblock {\em J. Comput. Phys.}, 53(3):443--455, 1984.

\bibitem{ZhangInp2012}
Z.M. Zhang.
\newblock Superconvergence points of polynomial spectral interpolation.
\newblock {\em SIAM J. Numer. Anal.}, 50(6):2966--2985, 2012.

\end{thebibliography}
\end{document}